\documentclass[12pt]{article}
\usepackage{amsmath}     
\usepackage{amssymb}     
\usepackage{amsthm}
\usepackage{amsfonts}
\usepackage{comment}
\usepackage{hyperref}
\usepackage[font=small]{caption}
\usepackage{tikz}
\usepackage{rotating}
\usepackage[showonlyrefs]{mathtools}
\mathtoolsset{showonlyrefs}
\numberwithin{equation}{section}

\def\re{\operatorname{Re}}
\def\im{\operatorname{Im}}
\def\arg{\operatorname{arg}}
\def\mult{\operatorname{\mu}}
\def\modulus{\operatorname{mod}}
\def\N{\mathbb{N}}
\def\Z{\mathbb{Z}}
\def\R{\mathbb{R}}
\def\C{\mathbb{C}}

\def\bC{\overline{\mathbb{C}}}
\def\D{\mathbb{D}}

\newtheorem{lemma}{Lemma}[section]
\newtheorem{theorem}{Theorem}[section]

\newtheorem{thmx}{Theorem}
\newtheorem{thmy}{Theorem}

\theoremstyle{remark}

\newtheorem*{remark0}{Remark}

\newtheorem{question}{Question}[section]

\begin{document}  
\title{On Bloch's ``Principle of topological continuity''}
\author{Walter Bergweiler and Alexandre Eremenko}
\date{}
\maketitle
\centerline{\emph{Dedicated to the memory of Larry Zalcman}}
\begin{abstract}
We discuss to what extent certain results about totally ramified values of
entire and meromorphic functions remain valid if one relaxes the hypothesis
that some value is totally ramified by assuming only that all islands over
some Jordan domain are multiple. 
In particular, we prove a result
suggested by Bloch which says that an entire function of order less than~$1$
has a simple island over at least one of two given Jordan domains with 
disjoint closures.

\medskip


\medskip

\end{abstract}
\section{Introduction}\label{intro}
About one hundred years ago, Andr\'e Bloch~\cite{Bloch1926} wrote a paper consisting mainly of
heuristic speculations based on two philosophical principles. 
In this paper, he anticipated several important results of 20th century geometric 
function theory -- and even today reading this paper can still be rewarding.

The first of these principles Bloch phrased as ``Nihil est in infinito quod non
prius fuerit in finito''. This principle has often been interpreted as meaning
that if all entire functions with a certain property are constant, then the 
family of functions which are holomorphic in some domain and have this property 
is normal. This is an important guideline in the 
theory of normal families; see \cite{Bergweiler2006,Zalcman1975,Zalcman1998} for a discussion. 
We will call it \emph{Bloch's normal family principle} in the sequel. 
Another interpretation of ``Nihil est in infinito quod non
prius fuerit in finito'' given by Bloch is that features
of transcendental entire and meromorphic functions are in some form already 
present in polynomials and rational functions.

The second principle that Bloch discusses he calls the \emph{principle of topological
continuity}. It says that certain true statements remain true if one modifies the data from
a metric point of view, but not from the topological point of view.

To illustrate this principle, he quotes the following theorem.
\begin{thmx} \label{bloch1}
Let $D$ be a domain and let $a_1,a_2,a_3\in\C$ be distinct.
Let $\mathcal F$ be the family of all functions holomorphic in $D$ 
which do not have a simple $a_j$-point, for all $j\in\{1,2,3\}$.
Then $\mathcal F$ is normal.
\end{thmx}
He then argues \cite[p.~87]{Bloch1926} that the principle of topological continuity should give
the following result.
\begin{thmy} \label{bloch2}
Let $D$ be a domain and let $D_1$, $D_2$ and $D_3$ be three disks in
$\C$ which have pairwise disjoint closures. 
Let $\mathcal F$ be the family of all functions holomorphic in $D$ 
such that there does not exist a domain $U$ with $\overline{U}\subset D$ 
such that $U$ is mapped univalently onto one of the disks $D_j$.
Then $\mathcal F$ is normal.
\end{thmy}

Theorem \ref{bloch1} was known when Bloch wrote this. 
Theorem \ref{bloch2} was not known then. It was first proved by Ahlfors six years
later, see \cite{Ahlfors1932a,Ahlfors1932b,Ahlfors1932c,Ahlfors1933},
his definite account being~\cite{Ahlfors1935}.
Theorem~\ref{bloch2}, and its generalization Theorem~\ref{bloch2c} below,
are among the principal results of his theory of covering surfaces
 which earned him one of the two first Fields medals in 1936.

Ahlfors did not use Theorem~\ref{bloch1} in his proof of Theorem~\ref{bloch2},
nor do the more recent proofs in~\cite{deThelin2005,Duval2014,Toki1957}.
Ahlfors does, however, give an interesting discussion of Bloch's principle of topological continuity
in~\cite[pp.~202--203]{Ahlfors1932c}.
A deduction of Theorem~\ref{bloch2} from Theorem~\ref{bloch1} -- and thus in some sense 
a confirmation of Bloch's principle of topological continuity -- was given in~\cite{Bergweiler1998}.
This deduction was based on a rescaling principle of Zalcman~\cite{Zalcman1975}.

Zalcman's lemma has become a major tool in the theory of normal families
by giving a rigorous formulation of \emph{Bloch's normal family principle}.
For example, it shows that Theorems \ref{bloch1} and \ref{bloch2} can be deduced from the
following corresponding results about entire functions.
\begin{thmx} \label{bloch1a}
Let $f$ be entire and let $a_1,a_2,a_3\in\C$ be distinct.
Suppose that all $a_j$-points are multiple, for all $j\in\{1,2,3\}$.
Then $f$ is constant.
\end{thmx}
\begin{thmy} \label{bloch2a}
Let $f$ be entire and let $D_1$, $D_2$ and $D_3$ be three disks in
$\C$ which have pairwise disjoint closures. 
Suppose that there does not exist $j\in\{1,2,3\}$ and a bounded domain
$U$ in $\C$ which is mapped univalently onto~$D_j$.
Then $f$ is constant.
\end{thmy}

As a second example for his principle of topological continuity Bloch
considers the following result.
\begin{thmx} \label{bloch1b}
Let $f\colon\C\to\C$ be a non-constant entire function of order less than $1$ and let $a_1,a_2\in\C$ be distinct.
Then there exists $j\in\{1,2\}$ such that $f$ has a simple $a_j$-point.
\end{thmx}
This theorem can be proved using Nevanlinna theory, and Bloch was aware of this
proof~\cite{Bloch1925}.
We will sketch a proof after Theorem~\ref{bloch2c} below.

Bloch argues \cite[p.~88]{Bloch1926} that Theorem~\ref{bloch1b} together with his
principle of topological continuity should yield the following result.
\begin{thmy} \label{bloch2b}
Let $f\colon\C\to\C$ be a non-constant entire function of order less than $1$ and let 
$D_1$ and $D_2$ be two disks in $\C$ with disjoint closures.
Then there exists $j\in\{1,2\}$ and a domain
$U$ in $\C$ which is mapped univalently onto~$D_j$.
\end{thmy}
So far as we know, there is no proof of Theorem~\ref{bloch2b} in the literature, neither
by Bloch nor by someone else.
We will give a proof of (a generalization of) Theorem~\ref{bloch2b} in section~\ref{proofs1},
using some ideas of Goldberg and Tairova~\cite{Goldberg1963}.

Theorems \ref{bloch1a} and \ref{bloch2a} are actually special cases of 
more general results. The generalization of Theorem~\ref{bloch1a} (which Bloch also knew) is the following result due to Nevanlinna~\cite[Chapitre~IV, no.~51]{Nevanlinna1929}.
\begin{thmx} \label{bloch1c}
Let $f\colon\C\to\bC:=\C\cup\{\infty\}$ be meromorphic and $q\in\N$.
Let $a_1,\dots,a_q\in\C$ be distinct and $m_1,\dots,m_q\in\N\cup\{\infty\}$.
If, for all $j\in\{1,\dots,q\}$,
all $a_j$-points have multiplicity at least $m_j$,
then
\begin{equation} \label{a1}
\sum_{j=1}^q \left(1-\frac{1}{m_j}\right) \leq 2
\end{equation}
or $f$ is constant.
\end{thmx}
Here $m_j=\infty$ means that $f$ does not take the value $a_j$ at all, and we 
put $1/m_j=1/\infty=0$ in this case.

To state the corresponding generalization of Theorem~\ref{bloch2a} we introduce some terminology.
Let $f\colon\C\to\bC$ be meromorphic and let $D$ be a Jordan domain in $\bC$.
A connected component $U$ of $f^{-1}(D)$ is called an \emph{island} of $f$ over $D$ if it is
bounded and simply-connected.
Then $f\colon U\to D$ is a proper mapping. The degree of this proper mapping is called the \emph{multiplicity}
of the island~$U$.
With this terminology we can phrase the generalization of Theorem~\ref{bloch2a}.
It is due to Ahlfors, who called it ``Scheibensatz'' \cite[p.~190]{Ahlfors1935}.
\begin{thmy} \label{bloch2c}
Let $f\colon\C\to\bC$ be meromorphic and $q\in\N$.
Let $D_1,\dots,D_q$ be Jordan domains in $\bC$ with 
pairwise disjoint closures  and let $m_1,\dots,m_q\in\N\cup\{\infty\}$.
If, for all $j\in\{1,\dots,q\}$, all islands over $D_j$ have multiplicity at least~$m_j$,
then \eqref{a1} holds or $f$ is constant.
\end{thmy}
Theorems \ref{bloch1a} and \ref{bloch2a} are obtained from 
Theorems \ref{bloch1c} and \ref{bloch2c} by choosing $q=4$, $m_1=m_2=m_3=2$, $m_4=\infty$
and $a_4=\infty$.

There are also normal family analogues of
Theorems \ref{bloch1c} and \ref{bloch2c} according to Bloch's normal
family principle. These can be obtained from 
Theorems \ref{bloch1c} and \ref{bloch2c} via Zalcman's lemma.

Suppose now that we have equality in \eqref{a1}; that is, 
\begin{equation} \label{a4}
\sum_{j=1}^q \left(1-\frac{1}{m_j}\right) = 2.
\end{equation}
Then, apart from permutation of the $m_j$, we have one of the following six cases:
\begin{itemize}
\item[$(i)$] $q=2$, $(m_1,m_2)=(\infty,\infty)$.
\item[$(ii)$] $q=3$, $(m_1,m_2,m_3)=(2,2,\infty)$.
\item[$(iii)$] $q=4$, $(m_1,m_2,m_3,m_4)=(2,2,2,2)$.
\item[$(iv)$] $q=3$, $(m_1,m_2,m_3)=(2,3,6)$.
\item[$(v)$] $q=3$, $(m_1,m_2,m_3)=(2,4,4)$.
\item[$(vi)$] $q=3$, $(m_1,m_2,m_3)=(3,3,3)$.
\end{itemize}
Selberg~\cite[Satz~II]{Selberg1928} determined all transcendental
meromorphic functions $f$ of finite order such that
$f$ has only finitely many $a_j$-points of multiplicity less than~$m_j$,
with the $m_j$ chosen such that~\eqref{a4} holds.
He used this to determine the possible orders of these functions~\cite[Satz~IV--VII]{Selberg1928}.
Let $\rho(f)$ denote the order of a meromorphic function~$f$;
see~\cite[Chapter~2, Section~1]{Goldberg2008} for the definition of the order 
and other concepts from the theory of meromorphic functions used in the sequel.

We briefly summarize Selberg's reasoning and results.
First, in case $(i)$ we assume that $(a_1,a_2)=(0,\infty)$
and find that the function $f$ has the form $f=Qe^P$ with a rational
function $Q$ and a polynomial~$P$.
Thus $\rho(f)\in\N$ in this case.

In case $(ii)$ we assume that $a_1=1$, $a_2=-1$ and $a_3=\infty$.
Then 
\begin{equation} \label{a2}
R(z)=\frac{f'(z)^2}{f(z)^2-1}
\end{equation}
has poles only at the simple $\pm 1$-points of $f$. 
Thus $R$ has only finitely many poles. 
Since $f$ has finite order, the lemma on the logarithmic
derivative~\cite[Chapter~3, Section~1]{Goldberg2008}
yields that Nevanlinna's proximity functions $m(r,R)$ satisfies 
$m(r,R)=O(\log r)$.  We conclude that $R$ is a rational function.
Hence $f$ has the form 
\begin{equation} \label{a2a}
f(z)=\cos\!\left( \int \! \sqrt{R(z)}dz\right) .
\end{equation}
This implies that $\rho(f)$ is an integer multiple of $1/2$.
(Equation~\eqref{a2a} and its implication for $\rho(f)$ were actually already
obtained by Valiron~\cite[p.~77]{Valiron1923}, before Selberg's work
and in fact before Nevanlinna developed his theory.)

If all $\pm 1$-points are multiple, then $R$ has no poles and is thus
a polynomial. In this case we find that $\rho(f)\geq 1$.
Note that this proves Theorem~\ref{bloch1b}.

In cases $(iii)$--$(vi)$, assume that all $a_j$ are in $\C$ and that $M$ is the least 
common multiple of the $m_j$. We then find that
\begin{equation} \label{a3}
R(z)=\frac{f'(z)^M}{\prod_{j=1}^q (f(z)-a_j)^{(m_j-1)M/m_j}}
\end{equation}
is a rational function and $f$ has the form 
\begin{equation} \label{a3a}
f(z)=E\!\left( \int \! \sqrt[M]{R(z)}dz\right) .
\end{equation}
where $E$ is an elliptic function.

We only state the conclusions about the order of $f$ that Selberg drew from this.
\begin{thmx} \label{selberg}
Let $f\colon\C\to\bC$ be a transcendental 
meromorphic function of finite order, $q\in\N$,
$a_1,\dots,a_q\in\bC$ distinct and $m_1,\dots,m_q\in\N\cup\{\infty\}$.
Suppose that~\eqref{a4} holds 
so that we have one of the cases $(i)$--$(vi)$ listed above.

If, for all $j\in\{1,\dots,q\}$, all but finitely many $a_j$-points of
$f$ have multiplicity at least~$m_j$,
then $\rho(f)\in\N$ in case $(i)$,
$2\rho(f)\in\N$ in case~$(ii)$,
$\rho(f)\in\N_0$ in case~$(iii)$,
$3\rho(f)\in\N_0$ in case~$(iv)$,
$2\rho(f)\in\N_0$ in case~$(v)$ 
and
$3\rho(f)/2\in\N_0$ in case~$(vi)$.

If all (and not only all but finitely many) $a_j$-points
have multiplicity at least~$m_j$,
then in addition $\rho(f)\geq 1$ in cases $(i)$ and $(ii)$ while $\rho(f)\geq 2$ in
cases $(iii)$--$(vi)$.
\end{thmx}
The question that motivated this paper is whether there exists a Theorem~\ref{selberg}$'$ which 
corresponds to Theorem~\ref{selberg} in the same way that 
Theorems~\ref{bloch2}, \ref{bloch2a}, \ref{bloch2b} and  \ref{bloch2c}
correspond to
Theorems~\ref{bloch1}, \ref{bloch1a}, \ref{bloch1b} and  \ref{bloch1c}.
In other words, we ask
to what extent the conclusion of Theorem~\ref{selberg}
remains valid if instead of the hypotheses of Theorem~\ref{bloch1c} we assume the 
hypotheses of Theorem~\ref{bloch2c}.
Note that Theorem~\ref{bloch2b}, which was envisaged by Bloch using his principle of 
topological continuity, says that this holds in case $(ii)$ if $f$ is entire.
In contrast, we will see that the corresponding result does not hold in 
cases $(i)$ and~$(iii)$.

The following result corresponds to case $(i)$ of Theorem~\ref{bloch1c}.

\begin{theorem}\label{thm1}
Let $f$ be a transcendental entire function and let $D$ be a Jordan domain in~$\C$.
If $f$ has only finitely many islands over $D$, then $\rho(f)\geq 1/2$.

Conversely, for every $\rho\in[1/2,\infty)$ and every Jordan domain $D$
there exists an entire function $f$ of order $\rho$
such that $f$ has no island over $D$.
\end{theorem}

Our next result corresponds to case $(ii)$ of Theorem~\ref{bloch1c}.
It has Theorem~\ref{bloch2b} as a corollary.
Here and in the following we will denote the multiplicity of an island $U$
over some Jordan domain by $\mult(U)$.

\begin{theorem}\label{thm2}
Let $f$ be a transcendental entire function of finite order
and let $D_1$ and $D_2$ be Jordan domains in~$\C$ with disjoint closures.
Suppose that $f$ has only finitely many simple islands over $D_1$ and~$D_2$.
Let $N$ be the number (counting multiplicity) of critical points $c$ of $f$ 
such that $f(c)\notin D_1\cup D_2$.
Put 
\begin{equation} \label{a5}
p= 2N+2+\sum_U (\mult(U)-2),
\end{equation}
where the sum is taken over all islands $U$ over $D_1$ or $D_2$.
Then $1\leq p<\infty$ and there exists $c>0$ such that 
\begin{equation} \label{a6}
\log M(r,f)\sim c\, r^{p/2}
\end{equation}
as $r\to\infty$.
In particular, $\rho(f)=p/2$.
\end{theorem}
\begin{remark0}
In the proof of Theorem~\ref{thm2} the hypothesis that $f$ is entire and
of finite order will be used only to conclude that $f^{-1}$ has finitely many
transcendental singularities. So actually we prove a more general theorem:

{\em Let $U$ be a simply connected domain in $\C$ and let 
$f\colon U\to\C$ be holomorphic.
Suppose that $f^{-1}$ has only finitely many transcendental singularities over $\bC$
and that $f$ has
only finitely many simple islands over two Jordan domains in $\C$ with disjoint closures.
Then $U=\C$, and $f$ is either a polynomial or an entire function satisfying~\eqref{a6}.
}
\end{remark0}

Under the hypothesis that there exist $a_1,a_2\in\C$ such that $f$ has 
only finitely many simple $a_j$-points, the conclusions of Theorem~\ref{thm2}
and the above remark were obtained by 
Goldberg and Tairova~\cite{Goldberg1963}.
Their proof was based on topological arguments and it probably can be 
extended to a proof of~Theorem~\ref{thm2}.

Note that we have $\mult(U)\geq 1$ and hence $\mult(U)-2\geq -1$
for all islands~$U$, with $\mult(U)-2=-1$
only for the at most finitely many simple islands.
Thus $p<\infty$ implies in particular that $N<\infty$ and that $\mult(U)>2$ for at 
most finitely many islands~$U$.

Note also that if all islands over $D_1$ and $D_2$ are multiple, then $\mult(U)\geq 2$ for 
all islands $U$ and thus $p\geq 2$.
Thus Theorem~\ref{bloch2b} follows from Theorem~\ref{thm2}.

Theorems~\ref{thm1} and~\ref{thm2} concern entire functions. The analogous
results do not
hold if instead of entire functions we consider meromorphic functions which have no island 
over a domain containing~$\infty$.
\begin{theorem}\label{thm3}
Let $D_1$ and $D_2$ be Jordan domains in~$\bC$ with disjoint closures.
Then, given $\rho\in[0,\infty)$, there exists a meromorphic function $f$ of order $\rho$ 
which has no island over $D_1$ and $D_2$.
\end{theorem}
\begin{theorem}\label{thm4}
Let $D$ be a  Jordan domain in~$\bC$ and let
$a_1,a_2\in \overline{\C}\setminus D$ be distinct.
Then, given $\rho\in[0,\infty)$, there exists a meromorphic function $f$ of order $\rho$ 
which has no island over $D$ and for which all $a_1$-points and $a_2$-points are multiple.
\end{theorem}
The functions $f$ constructed in the proofs of Theorems~\ref{thm3} and~\ref{thm4}
have the property that
\begin{equation} \label{a6b}
T(r,f)\geq c\,(\log r)^2
\end{equation}
for some $c>0$ and all large~$r$.
\begin{question}\label{qu1}
Do we have~\eqref{a6b} for every meromorphic function $f$ satisfying the 
conditions of Theorems~\ref{thm3} or~\ref{thm4}?
\end{question}

The following result says that in case $(iii)$
there is no analogue of Theorem~\ref{selberg} if instead of multiple $a_j$-points 
one considers multiple islands over certain Jordan domains.
In fact, it suffices to replace $a_j$-points by multiple islands over some Jordan domain for one~$j$.
\begin{theorem}\label{thm5}
Let $a_1,a_2,a_3\in \overline{\C}$ be distinct and $\rho\in(0,\infty)$.
Then there exists a meromorphic function $f$ of order $\rho$ and a Jordan
domain $D$ whose closure is contained in $\bC\setminus\{a_1,a_2,a_3\}$
such that all $a_j$-points of $f$ are multiple for $j\in\{1,2,3\}$
and such that $f$ has no simple island over $D$.
\end{theorem}
\begin{question}\label{qu2}
Does the conclusion of Theorem~\ref{thm5} also hold for $\rho=0$?
\end{question}
Theorem~\ref{thm5} deals with case $(iii)$. We do not know whether there
are analogous results for the cases $(iv)$--$(vi)$.
\begin{question}\label{qu4}
In cases $(iv)$--$(vi)$,
is there an analogue of Theorem~\ref{selberg} if instead of multiple $a_j$-points 
we consider multiple islands over Jordan domains $D_j$ with disjoint closures?
\end{question}
The answer may depend on whether we replace multiple $a_j$-points by multiple islands
over the $D_j$ for all $j\in\{1,\dots,q\}$ or only for some~$j$.

There exist functions of order $0$ for which
all but finitely many islands over $D_j$ have multiplicity~$m_j$.
In fact, this is the case already in the situation of Theorem~\ref{selberg}.
Given $a_j$ and $m_j$ as there there exists a function of order $0$ for which
all but finitely many $a_j$-points have multiplicity at least $m_j$.
Such a function $f$ was considered already by Teichm\"uller~\cite[p.~734]{Teichmueller1944}.
It satisfies 
\begin{equation} \label{a6a}
T(r,f)\sim c\,(\log r)^2
\end{equation}
as $r\to \infty$, for some $c>0$.
This leads to the following question analogous to Question~\ref{qu1}.
\begin{question}\label{qu3}
Let $f\colon\C\to\bC$ be meromorphic and $q\in\N$.
Let $D_1,\dots,D_q$ be Jordan domains in $\bC$ with 
pairwise disjoint closures  and let $m_1,\dots,m_q\in\N\cup\{\infty\}$
satisfy~\eqref{a4}.
Suppose that, for all $j\in\{1,\dots,q\}$, all but finitely many 
islands over $D_j$ have multiplicity at least~$m_j$.
Does there exists $c>0$ such that~\eqref{a6b} holds for all large~$r$?
\end{question}

\section{Results used in the proofs}\label{results_used}
We begin with a classical result of Wiman \cite[Chapter~5, Theorem~1.3]{Goldberg2008}.
\begin{lemma}\label{la3}
Let $f$ be a non-constant entire function. Suppose that there exists $C>0$ such that
$\min_{|z|=r}|f(z)|<C$ for all $r>0$. Then $\rho(f)\geq 1/2$.
\end{lemma}
The next result \cite[Chapter~5, Theorem~1.2]{Goldberg2008}
is a version of the Denjoy-Carleman-Ahlfors theorem.
\begin{lemma}\label{la1}
Let $f$ be an entire function of finite order  and let $C>0$.
Then $\{z\in\C\colon |f(z)|>C\}$ has at most $\max\{1,2\rho(f)\}$ 
connected components.
\end{lemma}
The following result \cite[Chapter~5, Theorem~1.4]{Goldberg2008} is a consequence of
Lemma~\ref{la1} and a result of Lindel\"of. It is also called  Denjoy-Carleman-Ahlfors theorem.
\begin{lemma}\label{la2}
Let $f$ be a non-constant entire function of finite order.
Then $f$ has at most $2\rho(f)$ finite asymptotic values.
\end{lemma}

We will use some results about quasiconformal and quasiregular mappings.
We refer to~\cite{Lehto1973} for the definition and basic properties,
noting that quasiregular mappings are called \emph{quasiconformal functions} there.
Let $D$ be a domain and $f\colon D\to\bC$ be quasiregular. We use the notation
\begin{equation}\label{c1}
\mu_f(z)=\frac{f_{\overline{z}}(z)}{f_z(z)},
\quad
K_f(z)=\frac{1+|\mu_f(z)|}{1-|\mu_f(z)|}
\quad\text{and}\quad
K(f)=\sup_{z\in D}|K_f(z)| .
\end{equation}
The H\"older continuity of quasiconformal mappings~\cite[\S~II.4.2]{Lehto1973}
yields the following result.
\begin{lemma}\label{la4}
Let $\phi\colon\C\to\C$ be quasiconformal. Then $|\phi(z)|=O(|z|^{K(\phi)})$ 
as $|z|\to\infty$.
\end{lemma}
A basic result~\cite[\S~V.1]{Lehto1973} in the theory of quasiconformal mappings
says that there exist
quasiconformal mappings with prescribed dilatation. One consequence of this 
is the following result.
\begin{lemma}\label{lemma-qr}
Let $f\colon\C\to\bC$ be quasiregular. Then there exists a quasiconformal mapping
$\phi\colon\C\to\C$ such that $f\circ\phi$ is meromorphic.
\end{lemma}
The next result is known as the Teich\-m\"uller-Wit\-tich-Belinskii theorem \cite[\S~V.6]{Lehto1973}.
\begin{lemma}\label{lemma-twb}
Let $\phi\colon\C\to\C$ be quasiconformal. Suppose that
\begin{equation}\label{c2}
\int_{\{z\in\C\colon |z|>1\}} \frac{K_\phi(z)-1}{x^2+y^2} dx\,dy < \infty.
\end{equation}
Then there exists $c\in\C\setminus\{0\}$ such that
\begin{equation}\label{c3}
\phi(z)\sim c z
\quad\text{as}\ z\to\infty .
\end{equation}
\end{lemma}
\begin{remark0}
The condition~\eqref{c2} is satisfied in particular if the 
set $A$ of all $z$ satisfying $|z|>1$ where $\phi$ is not conformal satisfies
\begin{equation}\label{c4}
\int_A \frac{dx\,dy}{|z|^2} <\infty.
\end{equation}
\end{remark0}

We denote the open and closed disk of radius $r$ around a point $a\in\C$ 
by $D(a,r)$ and $\overline{D}(a,r)$. We also put $\D:=D(0,1)$.

Beurling and Ahlfors~\cite{Ahlfors1956}
characterized the homeomorphisms of $\partial\D$ which admit a 
quasiconformal extension to~$\D$; see~\cite[\S~II.7.1]{Lehto1973}.
Such homeomorphisms are called quasisymmetric.

We will use the following sufficient condition for quasisymmetry. 
It is surely known, but we did not find a reference.
\begin{lemma}\label{la10}
Let $h\colon\partial\D\to\partial\D$ be an orientation-preserving homeomorphism.
Suppose that there exists a finite subset $A$ of $\partial \D$
such that $h$ is continuously differentiable with non-zero derivative
in $\partial \D\setminus A$.
Suppose also that for all $a\in A$ there exists $\gamma_a>0$ such that 
\begin{equation}\label{c3a}
\frac{|h(ae^{it})-h(a)|}{t^{\gamma_a}}
\end{equation}
has one-sided, non-zero limits as $t\to 0^\pm$.
Then $h$ is quasisymmetric.
\end{lemma}
\begin{proof}[Sketch of proof]
The restriction of $h$ to a closed arc which contains no point of $A$ 
is clearly quasisymmetric. 
To see that $h$ is quasisymmetric on an arc which has
a point of $A$ as one of its endpoints it is convenient to 
consider quasisymmetric mappings on $\R$ rather than~$\partial \D$.
Here this claim follows since $t\mapsto t^\gamma$ is quasisymmetric on $[0,1]$
and since the composition of quasisymmetric mappings is again 
quasisymmetric \cite[Theorem~9]{Kelingos1966}.
Finally, quasisymmetry in the union of these intervals (or arcs) follows 
from \cite[Theorem~3]{Kelingos1966}.
\end{proof}

Let $B\colon\D\to\D$ be a Blaschke product of degree $d\geq 2$ fixing~$0$.
Thus $B$ has the form 
\begin{equation}\label{c5}
B(z)=e^{i\alpha}z\prod_{k=1}^{d-1} \frac{z-a_k}{1-\overline{a_k}z}
\end{equation}
with $\alpha\in\R$ and $a_1,\dots,a_{d-1}\in\D$.

Suppose that $r\in(0,1)$ is such that all zeros of $B$ are in $D(0,r)$.
By a result of Walsh~\cite[Theorem~1]{Walsh1939}, all critical points of $B$ in $\D$
are also contained in $D(0,r)$.
Since $B\!\left(\overline{D}(0,r)\right)\subset D(0,r)$ by Schwarz's lemma we see that
$B^{-1}\!\left(D(0,r))\setminus \overline{D}(0,r)\right)$ is an annulus.

Branner and Fagella \cite[p.~163]{Branner2014} showed that there exists 
a quasiregular mapping $A\colon\D\to\D$  such that 
$A(z)=B(z)$ for $z\in \D\setminus B^{-1}(D(0,r))$ while
$A(z)=z^d$ for $z\in \overline{D}(0,r)$. Moreover, $A(z)\neq 0$ for $z\in\D\setminus\{0\}$.

They state this only for the case $d=2$, but their proof extends to the general case.
Note that in order to prove this result one has to define the mapping in the annulus
$B^{-1}(D(0,r))\setminus \overline{D}(0,r))$. The construction of this mapping is done
in~\cite[Exercise 2.3.3]{Branner2014} for arbitrary degree.

Branner and Fagella also note that we may choose $A$ to depend continuously on~$B$.
This implies that the dilatation $K(A)$ of $A$ also depends continuously on~$B$.
Restricting to a compact set of Blaschke products $B$ we find that the dilatation
of the corresponding quasiregular maps $A$ is uniformly bounded.

We summarize the above discussion in the following result.
\begin{lemma}\label{la7}
Let $0<r_1<r_2<1$ and 
let $B\colon \D\to\D$ be a Blaschke product of degree~$d\geq 2$.
Suppose that $B(0)=0$ and that the zeros of $B$ are contained in $D(0,r_1)$.
Then there exists a quasiregular mapping $A\colon\D\to\D$ and a neighborhood 
$W$ of $\partial \D$ such that 
$A(z)=z^d$ for $z\in D(0,r_2)$, $A(z)=B(z)$ for $z\in W\cap\D$ and
$A(z)\neq 0$ for $z\in\D\setminus\{0\}$.

Moreover, there exists a constant $C$ depending only on $r_1$, $r_2$ and $d$ such that
$A$ may be chosen to satisfy $K(A)\leq C$.
\end{lemma}

We will also need the following result.

\begin{lemma}\label{la9}
Let $U$ be a simply connected, unbounded domain in~$\C$
which is bounded by piecewise analytic curves.
Suppose that each disk $D(0,t)$ intersects only finitely many of these boundary curves.

Let $f\colon\overline{U}\to\C$ be a bounded, continuous function which is holomorphic
in $U$. Suppose that there exist $r,R>0$ such  that $|f(z)|>r$ for all $z\in U$ while  
$|f(z)|=r$ for all $z\in \partial U$ satisfying $|z|>R$.
Then there exists a curve $\gamma$ tending to $\infty$ in $U$ and 
$a\in\partial D(0,r)$ such that $f(z)\to a$ as $z\to\infty$ on~$\gamma$.
\end{lemma}
\begin{proof}
Let $\phi\colon \D\to U$ be a conformal mapping.
The hypotheses imply that the boundary of $U$ in $\bC$ is
locally connected. Thus $\phi$ has a continuous extension
$\phi\colon\overline{\D}\to\overline{U}\cup\{\infty\}$.

Put $E=\phi^{-1}(\infty)$. Then $E$ is a compact subset of $\partial\D$.
By a result of Beurling \cite[Theorem~9.19]{Pommerenke1992},
$E$ has logarithmic capacity zero.
Put $g=f\circ\phi$.
There exists an open arc $A$ containing $E$ such that 
$|g(z)|=r$ for $z\in A\setminus E$.
Noting that $|g(z)|>r$ and thus $g(z)\neq 0$ for $z\in\D$,
we deduce from the Schwarz reflection principle that $g$ can be extended to a function
holomorphic in $\D\cup (A\setminus E)\cup (\C\setminus\overline{\D})$.

Since $E$ has logarithmic capacity zero, it also
has analytic capacity zero; see, e.g., \cite[Proposition~3.5]{Zalcman1968}.
Since $g$ is bounded this yields that
$g$ has a holomorphic extension to $\D\cup A\cup (\C\setminus\overline{\D})$;
see \cite[Appendix~II]{Zalcman1968}.

Next we note that $U$ is unbounded and that $\infty$ is accessible in~$U$.
Thus $E\neq\emptyset$. Taking a point $\xi\in E$ we have
$\phi(t\xi)\to\infty$ as $t\to 1$.
The conclusion follows for the curve $\gamma\colon [0,1)\to\C$, $\gamma(t)=\phi(t\xi)$,
and $a=g(\xi)$.
\end{proof}

We will also use the following result about the growth of composite
meromorphic functions; see~\cite[Satz 2.3 and Satz 5.7]{Bergweiler1984}
and~\cite[Corollary~4]{Bergweiler1990}.
\begin{lemma}\label{la8}
Let $f$ be a meromorphic function and $g$ be an entire function. Then
\begin{equation} \label{c6}
\rho(g)\liminf_{r\to\infty}\frac{\log T(r,f)}{\log \log r}
\leq
\rho(f\circ g)
\leq
\rho(g)\limsup_{r\to\infty}\frac{\log T(r,f)}{\log \log r}.
\end{equation}
\end{lemma}
In~\cite[Satz 5.7]{Bergweiler1984} the left inequality in~\eqref{c6}
 is proved only under the additional hypothesis that
\begin{equation} \label{c7}
\limsup_{r\to\infty}\frac{\log T(r,f)}{\log \log r}<\infty.
\end{equation}
This additional hypothesis is removed in~\cite[Corollary~4]{Bergweiler1990}.
It will be satisfied, however, in our applications.
In fact, we will consider only the case where $(\log T(r,f))/\log\log r$ tends
to a finite limit.
In this case we deduce from~\eqref{c6} that
\begin{equation} \label{c8}
\rho(f\circ g)
=
\rho(g)\lim_{r\to\infty}\frac{\log T(r,f)}{\log \log r}.
\end{equation}

\section{Proofs of Theorems \ref{thm1}--\ref{thm5}} \label{proofs1}
\begin{proof}[Proof of Theorem \ref{thm1}]
Let $f$ be a transcendental entire function which has only finitely many
islands over some Jordan domain~$D$. Then there exists a connected 
component $U$ of $f^{-1}(D)$ which is unbounded. This implies that 
the minimum modulus $\min_{|z|=r}|f(z)|$ is bounded. 
Wiman's theorem (Lemma~\ref{la3}) now yields 
that $\rho(f)\geq 1/2$.

For the converse result, let $D$ be a Jordan domain in~$\C$. Without loss
of generality we may assume that $0\in D$. Thus there exists $\varepsilon>0$
such that $D(0,\varepsilon)\subset D$. 
For $\rho\in [1/2,1)$ we consider the function $g$ defined by
\begin{equation} \label{a7a}
g(z)=\prod_{n=1}^\infty \left(1-\frac{z}{n^{1/\rho}}\right)
\end{equation}
Then $g$ is a function of order $\rho$ 
which has only positive zeros and satisfies $g(x)\to 0$ as $x\to\infty$.
For $\rho\in (1/2,1)$ the last statement follows from, e.g., \cite[Theorem 4.1.8]{Boas1954}
while for $\rho=1/2$ it follows from the explicit representation
\begin{equation} \label{a7}
g(z)=\frac{\sin\sqrt{\pi z}}{\sqrt{\pi z}}
\end{equation}
We conclude that for sufficiently small $\delta>0$ the function 
$f:=\delta g$ has no island over $D(0,\varepsilon)$ and hence no island over~$D$.

To obtain functions of order $\rho\in [1,\infty)$ we write
$\rho=p\rho_0$ with $p\in\N$ and $\rho_0\in[1/2,1)$. Choosing $g$ 
of order $\rho_0$ and $\delta>0$ as above we find that the function $f$ given by
$f(z)=\delta g(z^p)$ has order $\rho$ and that $f$ has no island over~$D$.
\end{proof}
\begin{proof}[Proof of Theorem \ref{thm2}]
Let $r>0$ be such that the closures of $D_1$ and $D_2$ and all finite asymptotic 
values are contained in $D(0,r)$.
Note that by the Denjoy-Carleman-Ahlfors theorem (Lemma~\ref{la2})
there are only finitely many asymptotic values.
We also assume that $\partial D(0,r)$ contains no critical value.
This can be achieved since the set of critical values is countable.

We consider the following graph $\Gamma$ on the sphere~$\bC$.
It has two vertices, which we denote by $\times$ and $\circ$ and which lie
on $\partial D(0,r)$, and three edges, 
two of which are arcs on the circle $\partial D(0,r)$ connecting 
$\times$ and $\circ$, while the third edge is a crosscut of $D(0,r)$ 
connecting $\times$ and $\circ$ which separates $\overline{D_1}$ and $\overline{D_2}$
and which contains no critical or asymptotic value.
The components of the complement of the set of vertices and edges are called faces.
We then have three faces.
The face $\bC\setminus\overline{D}(0,r)$ will be denoted by $F_\infty$
and, for $j\in\{1,2\}$, the face containing $\overline{D_j}$ 
will be denoted by~$F_j$.

We consider the preimage $\Gamma^* = f^{-1}(\Gamma)$ of $\Gamma$.
It yields a partition of the plane into faces, edges and vertices.
It is similar to a \emph{line complex}; see~\cite[Chapter~7, Section~4]{Goldberg2008}.
One difference is that a line complex is always connected, while $\Gamma^*$ need not be
connected. We will see, however, that $\Gamma^*$ is connected if $r$ is chosen sufficiently large.

For our purposes only the topology of $\Gamma^*$ is relevant.
Thus we do not distinguish between the preimage $\Gamma^*$ and its 
image under a homeomorphism of the plane.
In figures like Figure~\ref{line-complex1}
we usually draw only a homeomorphic image 
of $\Gamma^*=f^{-1}(\Gamma)$, not the true preimage.
In such figures we use the labels~\textcircled{\footnotesize 1},
\textcircled{\footnotesize 2} and
\textcircled{\footnotesize $\infty$}
for the faces $F_1$, $F_2$ and $F_\infty$
as well as for their preimages in~$\Gamma^*$.
The same remark applies to the vertices $\times$ and~$\circ$.

Figure~\ref{line-complex1} shows $\Gamma$ and $\Gamma^*$ for a
function having only islands of multiplicity~$2$
over $D_1$ and $D_2$, and no critical or finite asymptotic values outside $D_1$ and~$D_2$.
An example is given by the sine or cosine function if $-1\in D_1$ and $1\in D_2$.
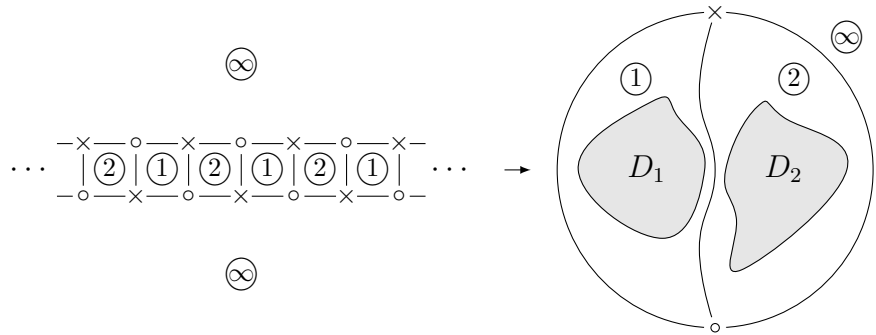
\begin{figure}[!htb]
\captionsetup{width=.85\textwidth}
\centering
\begin{tikzpicture}[scale=0.7,>=latex](-5,-5)(5,5)
\draw[->] (5,0) -- (5.5,0);
\draw (9,3) arc (90:270:3);
\draw (9,3) arc (90:-90:3);
\draw plot [smooth, tension=0.8] coordinates { (9.0,3.0) (8.7,1.5) (9.0,0.0) (8.7,-1.5) (9.0,-3.0) };
\filldraw[white] (9,-3) circle (0.2);
\filldraw[white] (9,3) circle (0.2);
\node at (9,3) {\footnotesize $\times$};
\draw (9,-3) circle (0.08);
\filldraw [gray!20] plot [smooth cycle, tension=0.8] coordinates { (6.4,0) (7.8,1.3) (8.3,1) (8.8,0) (8.0,-1.3) };
\draw plot [smooth cycle, tension=0.8] coordinates { (6.4,0) (7.8,1.3) (8.3,1) (8.8,0) (8.0,-1.3) };
\node[anchor=mid] at (7.7,0) {$D_1$};
\node[anchor=mid] at (7.5,1.7) {\textcircled{\footnotesize 1}};
\filldraw [gray!20] plot [smooth cycle, tension=0.8] coordinates { (9.2,0) (9.8,1.2) (10.3,1) (11.5,0) (9.5,-1.9) (9.4,-0.9) };
\draw plot [smooth cycle, tension=0.8] coordinates { (9.2,0) (9.8,1.2) (10.3,1) (11.5,0) (9.5,-1.9) (9.4,-0.9) };
\node[anchor=mid] at (10.3,0) {$D_2$};
\node[anchor=mid] at (10.5,1.7) {\textcircled{\footnotesize 2}};
\node[anchor=mid] at (11.5,2.5) {\textcircled{\footnotesize $\infty$}};
\node at (-4,0) {$\cdots$};
\node at (4,0) {$\cdots$};
\foreach \x in {-3,-1,...,3}
{
\node at (\x,0.5) {\footnotesize$\times$};
\node at (\x,-0.5) {\footnotesize$\circ$};
}
\foreach \x in {-2,0,2}
{
\node at (\x,0.5) {\footnotesize$\circ$};
\node at (\x,-0.5) {\footnotesize$\times$};
\node[anchor=mid] at (\x-0.5,0) {\textcircled{\footnotesize 2}};
\node[anchor=mid] at (\x+0.5,0) {\textcircled{\footnotesize 1}};
}
\foreach \x in {-3,-2,...,2}
{
\draw[-] (\x+0.2,0.5) -- (\x+0.8,0.5);
\draw[-] (\x+0.2,-0.5) -- (\x+0.8,-0.5);
\draw[-] (\x,-0.3) -- (\x,0.3);
}
\node[anchor=mid] at (0,2) {\textcircled{\footnotesize $\infty$}};
\node[anchor=mid] at (0,-2) {\textcircled{\footnotesize $\infty$}};
\draw[-] (3,-0.3) -- (3,0.3);
\draw[-] (-3.5,-0.5) -- (-3.2,-0.5);
\draw[-] (-3.5,0.5) -- (-3.2,0.5);
\draw[-] (3.5,-0.5) -- (3.2,-0.5);
\draw[-] (3.5,0.5) -- (3.2,0.5);
\end{tikzpicture}
\caption{The graphs $\Gamma$ (right) and $\Gamma^*$ (left).}
\label{line-complex1}
\end{figure}

Figure~\ref{line-complex2} shows $\Gamma^*$ for a
function having one simple island and one island of multiplicity $4$
over $D_1$, one island of multiplicity $3$ over $D_2$, while all 
other islands over $D_1$ and $D_2$ have multiplicity~$2$.

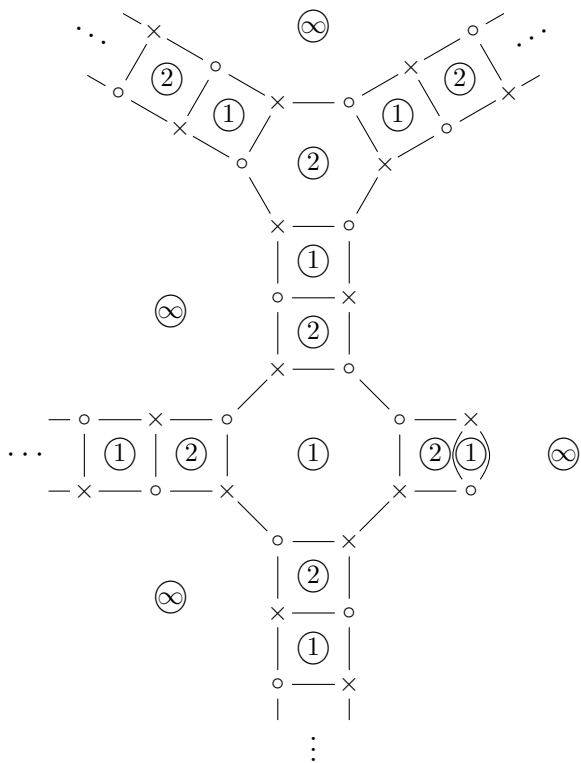
\begin{figure}[!htb]
\captionsetup{width=.85\textwidth}
\centering
\begin{tikzpicture}[scale=0.95,>=latex](-5,-5)(5,5)
\foreach \x in {-3.207,-2.207,1.207}
{
\draw[-] (\x+0.2,0.5) -- (\x+0.8,0.5);
\draw[-] (\x+0.2,-0.5) -- (\x+0.8,-0.5);
}
\draw[-] (-3.407,0.5) -- (-3.707,0.5);
\draw[-] (-3.407,-0.5) -- (-3.707,-0.5);
\node at (-4,0) {$\cdots$};
\foreach \x in {-3.207,-2.207,-1.207,1.207}
{
\draw[-] (\x,-0.3) -- (\x,0.3);
}
\foreach \y in {-3.207,-2.207,1.207,2.207}
{
\draw[-] (-0.5,\y+0.2) -- (-0.5,\y+0.8);
\draw[-] (0.5,\y+0.2) -- (0.5,\y+0.8);
}
\foreach \y in {-3.207,-2.207,-1.207,1.207,2.207,3.207,4.939}
{
\draw[-] (-0.3,\y) -- (0.3,\y);
}
\draw[-] (-0.5,-3.407) -- (-0.5,-3.707);
\draw[-] (0.5,-3.407) -- (0.5,-3.707);
\node at (0,-4) {$\vdots$};
\draw[-] (-0.5,-1.207) -- (-1.207,-0.5);
\draw[-] (0.5,-1.207) -- (1.207,-0.5);
\draw[-] (0.5,1.207) -- (1.207,0.5);
\draw[-] (-0.5,1.207) -- (-1.207,0.5);
\draw[-] (-0.5,3.207) -- (-1.,4.073);
\draw[-] (0.5,3.207) -- (1.,4.073);
\draw[-] (0.5,4.939) -- (1.,4.073);
\draw[-] (-0.5,4.939) -- (-1.,4.073);
\draw[-] (-0.5,4.939) -- (-1.366,5.439);
\draw[-] (-1.366,5.439) -- (-2.232,5.939);
\draw[-] (0.5,4.939) -- (1.366,5.439);
\draw[-] (1.366,5.439) -- (2.232,5.939);
\draw[-] (-1.,4.073) -- (-1.866,4.573);
\draw[-] (1.,4.073) -- (1.866,4.573);
\draw[-] (1.866,4.573) -- (2.732,5.073);
\draw[-] (-1.866,4.573) -- (-2.732,5.073);
\draw[-] (1.866,4.573) -- (1.366,5.439);
\draw[-] (-1.866,4.573) -- (-1.366,5.439);
\draw[-] ((2.732,5.073) -- (2.232,5.939);
\draw[-] ((-2.732,5.073) -- (-2.232,5.939);
\draw[-] (2.232,5.939) -- (2.665,6.189);
\draw[-] (-2.232,5.939) -- (-2.665,6.189);
\draw[-] ((2.732,5.073) -- (3.165,5.323);
\draw[-] ((-2.732,5.073) -- (-3.165,5.323);
\foreach \x in {-2.707,0,2.207}
{
\node[anchor=mid] at (\x,0) {\textcircled{\footnotesize 1}};
}
\foreach \x in {-1.707,1.707}
{
\node[anchor=mid] at (\x,0) {\textcircled{\footnotesize 2}};
}
\foreach \y in {-2.707,2.707}
{
\node[anchor=mid] at (0,\y) {\textcircled{\footnotesize 1}};
}
\foreach \y in {-1.707,1.707,4.073}
{
\node[anchor=mid] at (0,\y) {\textcircled{\footnotesize 2}};
}
\node[anchor=mid] at (1.183,4.756) {\textcircled{\footnotesize 1}};
\node[anchor=mid] at (-1.183,4.756) {\textcircled{\footnotesize 1}};
\node[anchor=mid] at (2.049,5.256) {\textcircled{\footnotesize 2}};
\node[anchor=mid] at (-2.049,5.256) {\textcircled{\footnotesize 2}};
\node[rotate=30] at (3.088,5.856) {$\cdots$};
\node[rotate=-30] at (-3.088,5.856) {$\cdots$};
\draw plot [smooth, tension=0.8] coordinates { (2.207,0.5) (1.95,0) (2.207,-0.5) };
\draw plot [smooth, tension=0.8] coordinates { (2.207,0.5) (2.464,0) (2.207,-0.5) };
\filldraw[white] (-1.207,-0.5) circle (0.2);
\filldraw[white] (-0.5,-1.207) circle (0.2);
\filldraw[white] (0.5,-1.207) circle (0.2);
\filldraw[white] (1.207,-0.5) circle (0.2);
\filldraw[white] (1.207,0.5) circle (0.2);
\filldraw[white] (0.5,1.207) circle (0.2);
\filldraw[white] (-0.5,1.207) circle (0.2);
\filldraw[white] (-1.207,0.5) circle (0.2);
\filldraw[white] (2.207,-0.5) circle (0.2);
\filldraw[white] (2.207,0.5) circle (0.2);
\filldraw[white] (-0.5,3.207)circle (0.2);
\filldraw[white] (0.5,3.207)circle (0.2);
\filldraw[white] (-1,4.073)circle (0.2);
\filldraw[white] (1,4.073)circle (0.2);
\node at (-1.,4.073) {\footnotesize$\circ$};
\node at (1.,4.073) {\footnotesize$\times$};
\filldraw[white] (-0.5,4.939)circle (0.2);
\filldraw[white] (0.5,4.939)circle (0.2);
\filldraw[white] (1.866,4.573) circle (0.2);
\node at (1.866,4.573) {\footnotesize$\circ$};
\filldraw[white] (1.366,5.439) circle (0.2);
\node at (1.366,5.439) {\footnotesize$\times$};
\filldraw[white] (2.232,5.939) circle (0.2);
\node at (2.232,5.939) {\footnotesize$\circ$};
\filldraw[white] (2.732,5.073) circle (0.2);
\node at (2.732,5.073) {\footnotesize$\times$};
\filldraw[white] (-1.866,4.573) circle (0.2);
\node at (-1.866,4.573) {\footnotesize$\times$};
\filldraw[white] (-1.366,5.439) circle (0.2);
\node at (-1.366,5.439) {\footnotesize$\circ$};
\filldraw[white] (-2.232,5.939) circle (0.2);
\node at (-2.232,5.939) {\footnotesize$\times$};
\filldraw[white] (-2.732,5.073) circle (0.2);
\node at (-2.732,5.073) {\footnotesize$\circ$};
\foreach \x in {-2.207,2.207}
{
\node at (\x,0.5) {\footnotesize$\times$};
\node at (\x,-0.5) {\footnotesize$\circ$};
}
\foreach \x in {-3.207,-1.207,1.207}
{
\node at (\x,-0.5) {\footnotesize$\times$};
\node at (\x,0.5) {\footnotesize$\circ$};
}
\foreach \y in {-2.207,1.207,3.207,4.939}
{
\node at (-0.5,\y) {\footnotesize$\times$};
\node at (0.5,\y) {\footnotesize$\circ$};
}
\foreach \y in {-3.207,-1.207,2.207}
{
\node at (0.5,\y) {\footnotesize$\times$};
\node at (-0.5,\y) {\footnotesize$\circ$};
}
\node[anchor=mid] at (0,6) {\textcircled{\footnotesize $\infty$}};
\node[anchor=mid] at (-2,2) {\textcircled{\footnotesize $\infty$}};
\node[anchor=mid] at (-2,-2) {\textcircled{\footnotesize $\infty$}};
\node[anchor=mid] at (3.5,0) {\textcircled{\footnotesize $\infty$}};
\end{tikzpicture}
\caption{Example of a graph $\Gamma^*$.}
\label{line-complex2}
\end{figure}

Clearly, $\Gamma^*$ is a bipartite, properly embedded graph.
(Here ``properly embedded'' means that it not only lies in the plane, but also that its
 vertices and edges do not accumulate to a point in the plane.)
We list some properties of this graph.
\begin{itemize}
\item[$(a)$]
Each vertex has degree~$3$ and lies on the boundaries of
three faces with labels \textcircled{\footnotesize 1},
\textcircled{\footnotesize 2} and~\textcircled{\footnotesize $\infty$}.
\item[$(b)$]
There are only finitely many digons
labeled \textcircled{\footnotesize 1} or~\textcircled{\footnotesize 2}.
(A digon is a face with only two boundary vertices and two boundary edges.)
\item[$(c)$]
Each face labeled \textcircled{\footnotesize $\infty$} is unbounded and 
there are only finitely many such faces.
\item[$(d)$]
Each face labeled \textcircled{\footnotesize 1} or \textcircled{\footnotesize 2}
is bounded.
\item[$(e)$]
If $r$ is large enough, then $\Gamma^*$ is connected.
\end{itemize}
Property $(a)$ is obvious from the definition of $\Gamma^*$.
To prove $(b)$ we note that for a digon $V$ labeled~\textcircled{\footnotesize $j$}
with $j\in\{1,2\}$ the mapping $f\colon V\to F_j$ is bijective and hence $V$ contains
a simple island over $D_j$.
Thus $(b)$ follows from the hypothesis that there are only finitely many simple islands over
$D_1$ and $D_2$.

To prove $(c)$ we note that in a face labeled \textcircled{\footnotesize $\infty$} 
the function $f$ is unbounded.  Thus such a face is unbounded and
by the Denjoy-Carleman-Ahlfors theorem (Lemma~\ref{la1}) there are 
only finitely many such faces.

To prove $(d)$, let $V$ be an unbounded face labeled~\textcircled{\footnotesize $j$}
with $j\in\{1,2\}$.
Then $\partial V$ contains an infinite chain $\cdots -\circ - \times-\circ - \times-\cdots$.
Since every vertex lies on the boundary of a face labeled~\textcircled{\footnotesize $\infty$},
and since there are only finitely many faces labeled~\textcircled{\footnotesize $\infty$},
there exist a face $V'$ labeled~\textcircled{\footnotesize $\infty$} such that
that this chain contains infinitely many vertices which lie on both $\partial V$ and $\partial V'$.
Let $v_1$ and $v_2$ be two such vertices and let $v_0$ be a vertex between
them. Then there exists a face $V''$ labeled~\textcircled{\footnotesize $\infty$} 
such  that $v_0\in\partial V''$. Connecting $v_1$ and $v_2$ by a crosscut in~$V'$
we see that $V''$ must intersect this crosscut. Thus $V''=V'$.
We conclude that there are infinitely many triplets of adjacent vertices which are on the 
boundary of both $V$ and $V'$. As the middle vertex of such a triplet has degree~$3$,
it must be connected to one of the other two vertices of the triplet by a double edge.
In other words, such a triplet leads to a digon.
Since there are only finitely many digons by $(b)$, this is a contradiction,
completing the proof of~$(d)$.

As a preparation for the proof of $(e)$, we note 
that there is a one-to-one correspondence between the
components of $\Gamma^*$ and the components of $f^{-1}\!\left(\overline{D}(0,r)\right)$.
In fact, given a component $C$ of $\Gamma^*$, the corresponding component of
$f^{-1}\!\left(\overline{D}(0,r)\right)$ is obtained by ``filling'' those
faces labeled~\textcircled{\footnotesize 1} or~\textcircled{\footnotesize 2}
whose boundaries are contained in~$C$.
Note that these faces are all bounded by~$(d)$.
Reversing this process, one obtains a component of $\Gamma^*$ from a component 
of $f^{-1}\!\left(\overline{D}(0,r)\right)$. 
We conclude from this that the number of components of $\Gamma^*$
is a non-increasing function of~$r$.

For the proof of $(e)$, as well as some subsequent arguments, 
it will be convenient to consider a graph $\Delta$ which in some sense
is dual to~$\Gamma^*$:
To each bounded face $V$ we associate a point $v\in V$.
(Recall that by $(c)$ and $(d)$ the bounded faces are those
 labeled~\textcircled{\footnotesize 1} or~\textcircled{\footnotesize 2}.)
These points $v$ are the vertices of~$\Delta$.
Two vertices are 
connected by an edge if the two faces of $\Gamma^*$ that contain these vertices
share a common edge in~$\Gamma^*$.
We take this edge in $\Delta$ to be in the union of the closures of the two faces in $\Gamma^*$,
crossing the edge in $\Gamma^*$ which separates these faces once.
So $f$ is bounded on the set of edges of~$\Delta$.
Figure~\ref{line-complex3} shows the graph  $\Delta$ corresponding to the graph $\Gamma^*$ in
Figure~\ref{line-complex2}.
(Again we only consider a homeomorphic image.)

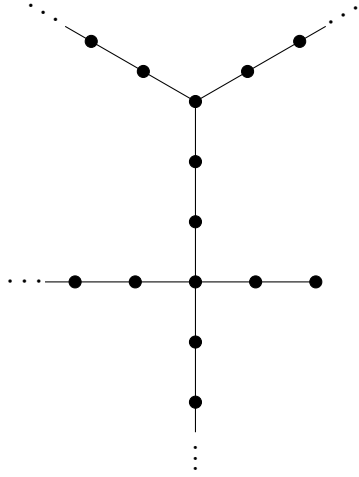
\begin{figure}[!htb]
\captionsetup{width=.85\textwidth}
\centering
\begin{tikzpicture}[scale=0.8,>=latex](-5,-5)(5,5)
\foreach \y in {-2,-1,...,3}
{
\filldraw (0,\y) circle (0.1);
}
\foreach \x in {-2,-1,1,2}
{
\filldraw (\x,0) circle (0.1);
}
\draw[-] (0,-2.5) -- (0,3);
\draw[-] (-2.5,0) -- (2,0);
\node at (-2.8,0) {$\cdots$};
\node at (0,-2.8) {$\vdots$};
\filldraw (-0.866,3.5) circle (0.1);
\filldraw (0.866,3.5) circle (0.1);
\filldraw (-1.732,4) circle (0.1);
\filldraw (1.732,4) circle (0.1);
\draw[-] (0,3) -- (2.165,4.25);
\draw[-] (0,3) -- (-2.165,4.25);
\node[rotate=30] at (2.501,4.45) {$\cdots$};
\node[rotate=-30] at (-2.501,4.45) {$\cdots$};
\end{tikzpicture}
\caption{The graph $\Delta$
corresponding to the graph $\Gamma^*$ in
Figure~\ref{line-complex2}.}
\label{line-complex3}
\end{figure}

Similarly as above we see that there is a one-to-one correspondence between the components 
of $\Delta$ and the components of $f^{-1}\!\left(\overline{D}(0,r)\right)$,
and hence to the components of~$\Gamma^*$. 
To pass from a component of $\Delta$ to a component of $f^{-1}\!\left(\overline{D}(0,r)\right)$
consider for a vertex $v$ the closure $\overline{V}$ of the face $V$ that contains~$v$.
Taking the union of the closures $\overline{V}$ over all $v$ in a component of $\Delta$
yields a component of~$f^{-1}\!\left(\overline{D}(0,r)\right)$.
This process can be reversed, so indeed there is the one-to-one correspondence mentioned. 
Bounded components of $\Gamma^*$ correspond to bounded components of~$\Delta$
and unbounded components of $\Gamma^*$ correspond to unbounded 
components of~$\Delta$.

Next we show that every component of $\Delta$ is a tree.
Otherwise there exists a closed curve and in fact a Jordan curve in~$\Delta$. We thus
have a Jordan curve $\gamma$ contained in the union of the faces
labeled~\textcircled{\footnotesize 1} or~\textcircled{\footnotesize 2}
which crosses each edge of $\Gamma^*$ at most once.
The interior of $\gamma$ contains some vertex, and this vertex 
is on the boundary of a face $V$ labeled~\textcircled{\footnotesize $\infty$}.
Since $V$ is unbounded, it must intersect~$\gamma$, which is a contradiction.
Thus every component of $\Delta$ is a tree.

For a face $V$ labeled~\textcircled{\footnotesize 1} or~\textcircled{\footnotesize 2}
the degree of the vertex $v$ of $\Delta$ such that $v\in V$ is given by $\mult(V)$.
We will also write $\mult(v)$ instead of $\mult(V)$. 
By hypothesis, there are only finitely many vertices of degree~$1$.
Since each bounded tree contains a vertex of degree~$1$, we conclude 
that $\Delta$ has only finitely many bounded components.
This implies that $\Delta$ also has an unbounded component.

We connect the finitely many bounded components of $\Delta$ by some paths 
to an unbounded component.
Then there exists $C>0$ such that $|f(z)|\leq C$ for $z$ on one of these paths.
This implies that $r<C$ and hence shows that if $r$ is chosen large enough at the 
beginning, then all components of $\Delta$ are unbounded.
It follows that all components of $\Gamma^*$ are unbounded.

To complete the proof of $(e)$, suppose that $\Gamma^*$ is disconnected.
Then there exists a face $V$ whose boundary contains two (unbounded) components of $\Gamma^*$.
We connect these two components by a crosscut $\gamma$ in~$V$. This crosscut separates 
$V$ into two domains $V_1$ and~$V_2$.

Since $\Gamma^*$ has no bounded components, $V$ and hence $V_1$ and $V_2$ are simply connected.
By~$(c)$, the face $V$ is labeled~\textcircled{\footnotesize $\infty$}
and hence $f$ is unbounded there.
We claim that $f$ is unbounded in each of the domains $V_1$ and~$V_2$.
In fact, suppose that $f$ is bounded in $V_j$ where $j\in\{1,2\}$.
Lemma~\ref{la9} yields that $f$ has an asymptotic value of modulus $r$ with
asymptotic path contained in~$V_j$.
This contradicts our assumption made at the beginning that all asymptotic values are
contained in $D(0,r)$.
Hence $f$ is unbounded in both $V_1$ and $V_2$.

Again there exists $C>0$ such that $|f(z)|\leq C$ for $z$ on the crosscut~$\gamma$.
Increasing $r$ to some value greater than $C$ thus increases the number of components
of $\{z\in\C\colon |f(z)|>r\}$.
Since, by Lemma~\ref{la1}, the number of these components is at most $\max\{1,2\rho(f)\}$
we conclude that if $r$ is large enough, then $\Gamma^*$ has only one component; that is,
$\Gamma^*$ is connected.
This yields $(e)$ and completes the proof of statements $(a)$--$(e)$.

As mentioned, a line complex is always connected, while $\Gamma^*$ 
and $\Delta$ need not be.
But $(e)$ says that this holds for large~$r$.

Let $p$ be the number of unbounded faces of~$\Gamma^*$.
Thus $p$ is the number of components of $\{z\in\C\colon |f(z)|>r\}$.
An argument similar to the one above about
the one-to-one correspondence between the components of~$\Gamma^*$
and the components of $\Delta$ 
yields that $p$ equals the number of complementary components of~$\Delta$.

Next we show that
\begin{equation} \label{a20}
p=2+\sum_{v} (\mult(v)-2),
\end{equation}
with the sum taken over all vertices $v$ of~$\Delta$.
The proof of~\eqref{a20} will only use
that $\Delta$ is an infinite properly embedded tree and that $\Delta$ has only finitely many 
vertices of degree~$1$.
To prove~\eqref{a20} suppose first that 
the number of vertices $v$ for which $\mult(v)\neq 2$ is finite.
If there are no such vertices, then
$\Delta$ is an infinite chain $\cdots -\circ - \times-\circ - \times-\cdots$.
In this case $\sum_{v} (\mult(v)-2)=0$ and thus~\eqref{a20} holds.

Let now $n\in\N$ and suppose that \eqref{a20} has been proved if the number of vertices $v$ 
with $\mu(v)\neq 2$ is less than $n$. Let $\Delta$ be a graph for which
there are $n$ such vertices.
If $\Delta$ has a vertex of degree $1$, we replace this vertex by an 
infinite half-chain $\circ - \times-\circ - \cdots$
(or $\times-\circ - \times \cdots$).
The new graph $\Delta'$ obtained has only $n-1$ 
vertices  $v$ with $\mu(v)\neq 2$ and both left and right side of \eqref{a20}
differ for $\Delta$ and $\Delta'$ by~$1$.
Thus \eqref{a20} holds for $\Delta$ since it holds for $\Delta'$ by induction hypothesis.

If $\Delta$ does not have a vertex of degree~$1$,
then $\Delta$ has a vertex of degree at least~$3$, and there exists such a vertex
$v_0$ with $\mult(v_0)\geq 3$ such that $v_0$ bounds $\mu(v_0)-1$ half-chains
$\circ - \times-\circ - \cdots$ (or $\times-\circ - \times \cdots$),
while all other vertices of $\Delta$ are on the remaining part of~$\Delta$.
Replacing the $\mu(v_0)-1$ half-chains by only one such half-chain yields 
a graph $\Delta'$ which has only $n-1$ vertices  $v$ with $\mu(v)\neq 2$.
The left and right side of \eqref{a20} differ for $\Delta$ and $\Delta'$ by~$\mu(v_0)-2$.
Again \eqref{a20} holds for $\Delta$ since it holds for $\Delta'$ by induction hypothesis.

This proves~\eqref{a20} if 
the number of vertices $v$ for which $\mult(v)\neq 2$ is finite.
However, minor modifications of the argument show that 
if the number of vertices $v$ for which $\mult(v)\neq 2$ is infinite,
then the number of complementary components of $\Delta$ is infinite.
Thus we also see that if $p$ is finite,
then the number of vertices $v$ with $\mult(v)\neq 2$ is finite.
This means that the sum in~\eqref{a20} is finite.

To pass from~\eqref{a20} to~\eqref{a5}, let
$v$ be a vertex of $\Delta$
and let $V$ be the face of $\Gamma^*$ containing~$v$.
Then $V$ is labeled~\textcircled{\footnotesize $j$} with $j\in\{1,2\}$ and it 
is an island over $F_j$.
Its boundary contains $\mult(V)$ vertices $\times$ and 
$\mult(V)$ vertices $\circ$. So it is an $n$-gon with $n=2\mult(V)$.

Also, $V$ contains at least one island over $D_j$.
Let $U_1,\dots,U_m$ be the islands over $D_j$ that are contained in~$V$.
Then
\begin{equation} \label{a8}
\mult(V)=\sum_{k=1}^m \mult(U_k).
\end{equation}
Let $N(V)$ be the number of critical points of $f$ contained in $V$ that are not mapped 
to~$D_j$.
Since $V$ contains $\mult(V)-1$ critical points and each $U_k$ contains
contains $\mult(U_k)-1$ critical points, we deduce from~\eqref{a8} that
\begin{equation} \label{a9}
N(V)=\mult(V)-1-\sum_{k=1}^m (\mult(U_k)-1) =m-1.
\end{equation}
This allows to rewrite~\eqref{a8} in the form 
\begin{equation} \label{a10}
\mult(V)-2=2m-2+\sum_{k=1}^m (\mult(U_k)-2)=2N(V)+\sum_{k=1}^m (\mult(U_k)-2).
\end{equation}
To obtain~\eqref{a5} we want to sum this over all~$V$.
Note that there may be (finitely many) vertices $v$ in $\Delta$ of 
degree~$2$ for which the corresponding face $V$ contains one critical
point and two simple islands over $D_j$.
For such a face $V$ both the left and right hand side of~\eqref{a10} are 
equal to~$0$.
Whether we include such a face or not will not affect the sum.
Similarly, the sum in~\eqref{a5} is unchanged if we remove this critical point and
the two simple islands.

Summing \eqref{a10} over all bounded faces $V$ and combining this with \eqref{a20} 
thus yields~\eqref{a5} with
$\sum_{V} N(V)$ instead of $N$, 
where the sum is taken over all bounded faces $V$ of~$\Gamma^*$.

To complete the proof of~\eqref{a5} we thus need to show that 
\begin{equation} \label{a21}
N=\sum_{V} N(V).
\end{equation}
In other words, we have to show that the unbounded
faces of $\Gamma^*$ contain no critical points of~$f$.
In order to do so, let $V$ be an unbounded face of~$\Gamma^*$.
Since $V$ is simply connected and bounded by a single curve, unbounded in both directions,
there exists a conformal mapping $\phi$ from $\D$ to $V$
such that $\phi$ has a continuous extension to $\overline{\D}\setminus\{1\}$,
mapping $\partial\D\setminus\{1\}$ to $\partial V$ and satisfying
$\phi(z)\to \infty$ as $z\to 1$.
Then 
\begin{equation} \label{a22}
u(z)=\log \frac{|f(\phi(z))|}{r}
\end{equation}
defines a positive harmonic function in $\D$  which extends 
continuously to $\overline{\D}\setminus\{1\}$, with $u(z)=0$ for 
$z\in \partial\D\setminus\{1\}$.
This yields that $u$ is a constant multiple of the Poisson kernel;
see, e.g., \cite[Theorem~6.19]{Axler2001}.
But this implies that $\log f\circ \phi$ has no critical points in~$\D$.
Hence $f$ has no critical point in~$V$.
This completes the proof of~\eqref{a21} and hence of~\eqref{a5}.

For $j\in\{1,2\}$ we choose $a_j\in D_j$ and a conformal mapping $\psi_j\colon F_j\to \D$
satisfying $\psi_j(a_j)=0$.
For a face $V$ labeled~\textcircled{\small $j$} we choose 
a conformal mapping $\tau_V\colon \D\to V$ with $\tau_V(0)\in f^{-1}(a_j)$.
Then $B:=\psi_j\circ f\circ \tau_V$ is a Blaschke product and $B(0)=0$.
The zeros of $B$ are contained in $\tau_V^{-1}(f^{-1}(D_j))$.
We will show that there exists $r_1\in (0,1)$,
depending only on the choice of the $F_j$ but not on~$V$,
such $\tau_V^{-1}(f^{-1}(D_j))$ and hence the zeros of $B$ are contained in $D(0,r_1)$.

In order to do so we note that since $p<\infty$, for all but finitely many faces $V$ 
labeled~\textcircled{\small $j$} there is  exactly one island $U$ of multiplicity
$2$ over $D_j$ contained in~$V$, but there are no further islands over $D_j$ contained in~$V$, and 
there are no critical points in $V\setminus U$.
For such a face $V$ and island $U$ we then have $f^{-1}(D_j)\cap V=U$
and the mapping $f\colon V\setminus \overline{U}\to F_j\setminus\overline{D_j}$
is a covering map of degree~$2$. Denoting by $\modulus(\Omega)$ the modulus of an annulus $\Omega$ 
we find that 
\begin{equation} \label{a11a}
\begin{aligned} 
\modulus\!\left( \D\setminus \overline{\tau_V^{-1}(f^{-1}(D_j))}\right)
&=
\modulus\!\left( V\setminus \overline{f^{-1}(D_j)}\right) 
\\ &
=
\modulus\!\left( V\setminus \overline{U}\right) 
=\frac12
\modulus\!\left( F_j\setminus \overline{D_j}\right).
\end{aligned} 
\end{equation}
Thus the modulus of $\D\setminus \overline{\tau_V^{-1}(f^{-1}(D_j))}$ is bounded below.
This implies that there exists $r_1\in (0,1)$ such that $\tau_V^{-1}(f^{-1}(D_j))\subset D(0,r_1)$
for all such faces~$V$. Increasing $r_1$ if necessary we may assume that this inclusion
also holds for the finitely many faces $V$ where 
$f\colon V\setminus \overline{f^{-1}(D_j)}\to F_j\setminus\overline{D_j}$
is not a covering of degree~$2$.

We now choose $r_2\in (r_1,1)$ and apply Lemma~\ref{la7} to $B$. With the function $A$ obtained from
this lemma we define $f_V\colon V\to F_j$ by $f_V=\psi_j^{-1}\circ A\circ \tau_V^{-1}$.
Thus $f_V$ is a quasiregular mapping having one $a_j$-point of multiplicity $\mu(V)$ and no other 
$a_j$-point in~$V$, and there exists a neighborhood $W$ of $\partial V$ with 
$f_V(z)=f(z)$ for $z\in V\cap W$.
Since there are only finitely many faces $V$ 
labeled~\textcircled{\footnotesize $j$} with $j\in\{1,2\}$
for which the degree of the mapping $f_V\colon V\to F_j$ is greater than~$2$,
Lemma~\ref{la7} also  yields that
 we may choose the mappings $f_V$ with uniformly bounded dilatation.

We now define a mapping $g\colon \C\to\C$ by putting $g(z)=f_V(z)$ if $z\in V$ for such a
face $V$, and $g(z)=f(z)$ otherwise.
Then $g$ is quasiregular.
By  Lemma~\ref{lemma-qr} there exists a quasiconformal 
homeomorphism $\phi\colon\C\to\C$ such that the mapping $h:=g\circ \phi$ is entire.
It follows that all except possibly finitely many $a_j$-points 
are multiple, for $j\in\{1,2\}$. Moreover, Lemma~\ref{la4} yields that $\rho(h)\leq K(\phi)\rho(f)<\infty$.

As in~\eqref{a2} we now consider
\begin{equation} \label{a12}
R(z)=\frac{h'(z)^2}{(h(z)-a_1)(h(z)-a_2)}
\end{equation}
and deduce from the lemma on the logarithmic derivative that $R$ is a rational function.
Assuming without loss of generality that $a_{1,2}=\pm 1$ we find as in~\eqref{a2a} that $h$ has the form
\begin{equation} \label{a13}
h(z)=\cos\!\left( \int \! \sqrt{R(z)}dz\right) .
\end{equation}
We saw above that $\{z\in\C\colon |f(z)|>r\}$ has $p$ components if $r$ is sufficiently large.
This implies that 
$\{z\in\C\colon |h(z)|>r\}$ has $p$ components for large~$r$.
Together with~\eqref{a13} this yields that 
\begin{equation} \label{a14}
\int \! \sqrt{R(z)}dz \sim \alpha z^{p/2}
\quad\text{and hence}\quad
R(z)\sim \beta z^{p-2}
\end{equation}
as $z\to\infty$, for certain $\alpha,\beta\in\C\setminus\{0\}$.
It follows that 
\begin{equation} \label{a15}
\log M(r,h)\sim \gamma r^{p/2}
\end{equation}
for some $\gamma>0$.

The function $g$ agrees with $f$ in $\C\setminus f^{-1}(F_1\cup F_2)$. Thus
$g$ is holomorphic in a neighborhood of any point $z\in\C$ for which $|g(z)|>r$.
Hence for large $r$ the function $\phi$ is conformal in a neighborhood of $z$ if $|h(z)|=|g(\phi(z))|>r$.
Thus the set of points in $\C\setminus\overline{\D}$ where $\phi$ is not conformal is contained in 
\begin{equation} \label{a16}
A:=\{z\in\C\colon |z|>1\ \text{and}\ |h(z)|<r\},
\end{equation}
provided $r$ is sufficiently large.
It can be deduced from~\eqref{a13} that $A$ satisfies~\eqref{c4}.
Lemma~\ref{lemma-twb} and the remark following it now imply that that there exists $a\in\C$ such that
$\phi(z)\sim az$
as $z\to\infty$.
Now~\eqref{a6} follows from~\eqref{a15}.
\end{proof}

\begin{proof}[Proof of Theorem \ref{thm3}]
Without loss of generality we may assume that $0\in D_1$ and $\infty\in D_2$.
Then there exists $\varepsilon>0$ such that 
$D(0,\varepsilon)\subset D_1$ and
$\{z\in\C\colon |z|>1/\varepsilon\}\cup\{\infty\}\subset D_2$.
We put $a=1+\varepsilon$ and 
\begin{equation} \label{b1}
f(z)=\prod_{k=1}^\infty \frac{1-z/a^k}{1+z/a^k} .
\end{equation}
It is easy to see that the infinite product converges and 
thus defines a function $f$ meromorphic in~$\C$.

For $x\geq a$ there exists $n\in\N$ and $\eta\in [0,1)$ such that $x=a^{n+\eta}$.
Hence
\begin{equation} \label{b2}
|f(x)|\leq\left|\frac{1-x/a^n}{1+x/a^n}\right|= \frac{a^\eta-1}{a^\eta+1}
\leq \frac{a-1}{a+1}<\varepsilon.
\end{equation}
This implies that $f$ has no island over $D(0,\varepsilon)$ and hence no
island over~$D_1$. An analogous argument shows that $f$ has no island over
$\{z\in\C\colon |z|>1/\varepsilon\}\cup\{\infty\}$ and hence no island over~$D_2$.

Standard arguments show that the function $f$ defined by~\eqref{b1} has order~$0$
and in fact that
\begin{equation} \label{d17}
T(r,f)\sim c\,(\log r)^2
\end{equation}
for some $c>0$ as $r\to\infty$.

This completes the proof for the case that $\rho=0$. 
To deal with the general case we note that if $g$ is any entire function,
then $f\circ g$ has no island over $D_1$ and $D_2$.
Since
\begin{equation} \label{d18}
\lim_{r\to\infty} \frac{\log T(r,f)}{\log\log r}=2
\end{equation}
by~\eqref{d17} we deduce from~\eqref{c8} that $\rho(f\circ g)=2\rho(g)$.
Thus we can achieve that $f\circ g$ has the preassigned order $\rho$
by choosing $g$ with $\rho(g)=\rho/2$.
\end{proof}

\begin{proof}[Proof of Theorem \ref{thm4}]
Without loss of generality we may assume that $a_{1,2}=\pm 2i$ and that $\infty\in D$.
Let $\delta\in (0,1)$ and $\beta\in (0,1/2)$ and let $U$ be the domain
which contains the imaginary axis and which is bounded by the 
curve 
\begin{equation} \label{d0}
\gamma\colon [0,\infty)\to\C, \quad \gamma(t)=\delta+e^{i\beta\pi}t,
\end{equation}
and the curves $-\gamma$, $\overline{\gamma}$ and $-\overline{\gamma}$ obtained 
from $\gamma$ by reflections.
The idea is to glue the restriction (of a modification) of $2 \sinh(\arcsin z)$ to 
the  domain $U$ and the restriction of the function given by~\eqref{b1} to a half-plane.

In order to do so we note that $\arcsin \gamma$ is an injective curve in the first 
quadrant which connects $\arcsin \delta$ with $\infty$. A computation shows 
that 
\begin{equation} \label{d1}
\arcsin\gamma(t)=\left(\frac{1}{2}-\beta\right)\pi -i(\log t +\log 2) +O\!\left(\frac{1}{t}\right)
\end{equation}
as $t\to\infty$.

Next we consider $G(z)=f(-z)$, with the function $f$ from~\eqref{b1}; that is,
\begin{equation} \label{d2}
G(z)=\prod_{k=1}^\infty \frac{1+z/a^k}{1-z/a^k} ,
\end{equation}
with some $a>1$ to be determined later.
Denoting by $\log G$ the branch of the logarithm with $\log G(0)=1$ we find that
\begin{equation} \label{d3}
\log G(it) = 
\sum_{k=1}^\infty \log \frac{1+it/a^k}{1-it/a^k} 
=2i 
\sum_{k=1}^\infty \arg\!\left(1+\frac{it}{a^k}\right)
=2i h(t)
\end{equation}
with
\begin{equation} \label{d4}
h(t):=
\sum_{k=1}^\infty \arctan\!\left(\frac{t}{a^k}\right). 
\end{equation}
We have 
\begin{equation} \label{d5}
\int_{1}^\infty \arctan\!\left(\frac{t}{a^s}\right) ds
\leq
h(t)
\leq
\int_{0}^\infty \arctan\!\left(\frac{t}{a^s}\right) ds .
\end{equation}
Now 
\begin{equation} \label{d6}
\int_{0}^\infty \arctan\!\left(\frac{t}{a^s}\right) ds
=\frac{1}{\log a}\int_0^t  \frac{\arctan u}{u}du \sim \frac{\pi}{2 \log a} \log t
\end{equation}
as $t\to\infty$.
Together with~\eqref{d5} this yields that 
\begin{equation} \label{d7}
h(t)= \frac{\pi}{2 \log a} \log t +O(1)
\end{equation}
as $t\to\infty$. It follows easily from~\eqref{d4} that $h$ is increasing and 
concave. This implies that~\eqref{d7} can be improved to 
\begin{equation} \label{d8}
h(t)= \frac{\pi}{2 \log a} \log t +\eta +o(1) 
\end{equation}
for some $\eta\in\R$ as $t\to\infty$.
Hence 
\begin{equation} \label{d9}
\log G(it)= i \frac{\pi}{\log a} \log t +2i\eta +o(1) 
\end{equation}
as $t\to\infty$ by~\eqref{d3}.

Let now $V$ be the domain bounded by the curves $\gamma$ and $\overline{\gamma}$
which contains the interval $(\delta,\infty)$. Thus 
\begin{equation} \label{d10}
V=\{z\in\C\colon |\arg (z-\delta)|<\beta\pi\}.
\end{equation}
Let $b>0$ and define 
\begin{equation} \label{d11}
H\colon V\to\C,\quad H(z)=e^b G((z-\delta)^{1/(2\beta)}) .
\end{equation}
It follows from~\eqref{d9} that
\begin{equation} \label{d12}
\log H(\gamma(t))=
b+\log G(it^{1/(2\beta)})=b+ i \left(\frac{\pi}{2\beta\log a} \log t +2\eta\right) +o(1) 
\end{equation}
as $t\to\infty$. 
This holds for any choice of $a$, $b$ and $\beta$. We choose 
$\beta=\pi/(2\log a)$ and $b=(1/2-\beta)\pi$.
Note that this still leaves the possibility to choose~$a$ later.
Then the right hand sides of~\eqref{d1} and~\eqref{d12}  have the same asymptotics 
as $t\to\infty$, apart from an additive constant.

Using interpolation it can now be shown that there exists
an odd quasiconformal mapping $\psi\colon U\to\C$, symmetric with respect to $\R$, 
which agrees with the arcsine in a neighborhood of the imaginary axis and
which satisfies
\begin{equation} \label{d13}
\psi(\gamma(t))=\log H(\gamma(t))\quad\text{for all}\ t\geq 0.
\end{equation}
Note that $\log H$ and hence $\psi$ map the curve~$\gamma$, and hence 
the boundary of~$V$, to the line $\{z\in\C\colon \im z=b\}$.

Next we consider the function $S\colon \{z\in\C\colon |\im z|\leq b\}\to\C$,
\begin{equation} \label{d14}
S(z)=
\begin{cases}
\displaystyle \frac{2}{b}(\im z+b)e^{z} -e^{-z}& \text{if}\ \displaystyle -b\leq \im z\leq -\frac{b}{2},  \\[2mm]
2 \sinh z & \text{if}\ |\im z|<b , \\[2mm]
\displaystyle e^z+\frac{2}{b}(\im z-b)e^{-z} & \text{if}\ \displaystyle \frac{b}{2}\leq \im z\leq b. 
\end{cases}
\end{equation}
Thus $S(z)=\pm e^{\pm z}$ if $\im z=\pm b$.
It is easy to see that $S$ is quasiregular.
Finally we define $F\colon\C\to\bC$ by 
\begin{equation} \label{d15}
F(z)=
\begin{cases}
S(\psi(z)) & \text{if}\ z\in \overline{U}, \\
H(z)  & \text{if}\ z\in V, \\
-H(-z)  & \text{if}\ z\in -V. 
\end{cases}
\end{equation}
Note that by~\eqref{d13} and~\eqref{d14} we $S(\psi(z))=\exp\psi(z)=H(z)$ for 
$z\in\partial V$. Thus $F$ defines a quasiregular mapping.
Lemma~\ref{lemma-qr} yields that there exists a quasiconformal homeomorphism
$\phi\colon\C\to\C$ such that $f:=F\circ \phi$ is meromorphic.

Noting that the hyperbolic sine has the totally ramified values $\pm i$ and recalling
that we have assumed that $a_{1,2}=\pm 2i$ we see that 
all $a_j$-points are multiple for $j\in\{1,2\}$.
A similar argument as in the proof of Theorem~\ref{thm3} shows that 
given $\varepsilon>0$, we can choose $a$ in~\eqref{d2} such that $f$ has no island
over $\{z\in\C\colon |z|>1/\varepsilon\}\cup\{\infty\}$.
Choosing $\varepsilon$ sufficiently small we conclude that $f$ has no island
over~$D$.

Finally, we have 
\begin{equation} \label{d16}
n(r,F)\sim\frac{1}{\beta\log a}\log r =\frac{2}{\pi}\log r
\quad\text{and}\quad
n\!\left(r,\frac{1}{F-a_j}\right)
\sim\frac{2}{\pi}\log r 
\end{equation}
as $r\to\infty$, for $j\in\{1,2\}$.
Lemma~\ref{la4} together with standard arguments
now shows that $f$ has order~$0$ an in fact that~\eqref{d18} holds.

This proves the theorem for $\rho=0$. As at the end of the proof 
of Theorem~\ref{thm3} we can use this to obtain the result for any
preassigned order $\rho\in (0,\infty)$ by considering $f\circ g$ instead of $f$ 
for an entire function $g$ satisfying $\rho(g)=\rho/2$.
\end{proof}

\begin{proof}[Proof of Theorem \ref{thm5}]
The idea behind the construction is due to  K\"unzi~\cite{Kuenzi1956}.
The details are somewhat different though.

An outline of the construction is as follows. 
We consider two elliptic function $g_1$ and $g_2$, both having periods $2$ and $2i\tau$, 
where $\tau>0$. 
We restrict $g_1$ and $g_2$ to the sectors
\begin{equation} \label{k5}
S_1:=\{z\in\C\colon |\im z|\leq \tau \re z\}
\quad\text{and}\quad
S_2:=\{z\in\C\colon |\im z|\leq -\tau \re z\}.
\end{equation}
We will modify  the $g_j$ near $\partial S_j$ to obtain quasiregular mappings $f_j\colon S_j\to\C$
satisfying
\begin{equation} \label{k6}
f_1(t(1\pm i\tau))=f_2(t(-1\pm i\tau)).
\end{equation}
With $\alpha:=\arctan\tau$ this yields that
\begin{equation} \label{k7}
f_0(z)= 
\begin{cases} 
f_1( z^{2\alpha/\pi}) & \text{if}\ \re z\geq 0,\\
f_2( -(-z)^{2\alpha/\pi}) & \text{if}\ \re z\leq 0,
\end{cases} 
\end{equation}
defines a quasiregular mapping $f_0\colon \C\to\bC$.
By Lemma~\ref{lemma-qr} there exists a quasiconformal mapping
$\phi\colon\C\to\C$ such that $F:=f_0\circ \phi$ is meromorphic.
The mapping $f$ we want to construct is then given by $f(z)=F(z^p)$
for some $p\geq 2$. 
Using Lemma~\ref{lemma-twb} we will see that $\rho(f)=4p\alpha/\pi$ so 
that we can achieve any preassigned positive order for~$f$.

Moreover, both $f_1$ and $f_2$ will have the critical values $a_1$, $a_2$ and $a_3$.
The fourth critical value of $f_1$ will be different from that of $f_2$. 
The domain $D$ will be such that it contains these fourth critical values.

We now come to the details of the construction.
Without loss of generality we may assume that $\{a_1,a_2,a_3\}=\{0,1,\infty\}$.
Let 
\begin{equation} \label{k}
R=\{x+iy\colon 0\leq x\leq 1,\, 0\leq y\leq\tau\}
\end{equation}
and let $g_1$ be a conformal mapping from the interior of $R$ onto the upper half-plane.
The mapping $g_1$ extends continuously to the boundary of $R$ and we may normalize
it to satisfy $g_1(0)=1$, $g_1(1)=\infty$ and $g_1(1+i\tau)=-1$. Then
$a:=g_1(i\tau)\in (-1,1)$.
The mapping $g_1$ can be extended by reflections to an elliptic function with
periods $2$ and $2i\tau$.

The mapping $g_1$ can be expressed in terms of the Weierstrass $\wp$-function
with these periods.
In fact, if $L$ is the fractional linear transformation satisfying 
$L(\infty)=1$, $L(e_1)=\infty$ and $L(e_3)=-1$, then $g_1=L\circ \wp$.

To define the quasiregular mapping $f_1\colon S_1\to\bC$
we put, for $m,n\in\Z$,
\begin{equation} \label{k0}
R_{m,n}=m+in\tau+R=\{x+iy \colon m\leq x\leq m+1,\, n\tau\leq y\leq (n+1)\tau\} .
\end{equation}
If $R_{m,n}\subset S_1$, we put $f_1(z)=g_1(z)$ for $z\in R_{m,n}$.
We also put $f_1(x)=g_1(x)$ for $0\leq x<1$.

It remains to define $f_1$ in $\Delta_{m,n}:=S_1\cap R_{m,n}$ for those $m,n\in\Z$
for which $R_{m,n}\not\subset S_1$, but the interior of $R_{m,n}$ intersects~$S_1$.
(This is the case if $m\geq 0$ and $n=m$ or $n=-m-1$.)
Note that $\Delta_{m,n}$ is a triangle for such $m$ and~$n$.
We begin by defining $f_1$ on $\Delta:=\Delta_{0,0}$.
In fact, we will first define $f_1$ on~$\partial \Delta$.
This will be done in such a way that
it can be extended quasiconformally to the interior of $\Delta$ using Lemma~\ref{la10}.

Given that $f_1$ is defined already on $\partial\Delta\cap S_1$, it remains
to define $f_1$ on $\partial\Delta\cap\partial S_1=\{t(1+i\tau)\colon 0\leq t\leq 1\}$.
To motivate the definition we note that
$g_1$ maps $\partial\Delta\cap\partial S_1$ to a curve in the upper 
half-plane which connects $1$ to $-1$. We want to define $f_1$ such that it
maps  $\partial\Delta\cap\partial S_1$
to the semicircle $\{e^{it}\colon 0\leq t\leq \pi\}$ which also connects
$1$ and~$-1$. The quasiconformal extension of $f_1$ will then map $\Delta$
to the domain $\Omega_1 :=\{z\in\C\colon \im z>0,\, |z|>1\}$. Thus we make the ansatz 
\begin{equation} \label{k1}
f_1(t(1+i\tau))=\exp (i\pi H(t))
\end{equation}
with a homeomorphism $H\colon [0,1]\to [0,1]$ satisfying $H(0)=0$ and $H(1)=1$. 

We want to choose $H$ such that the resulting mapping 
$f_1\colon \partial \Delta\to\partial \Omega_1$ 
can be extended quasiconformally to~$\Delta$.
Let $\sigma\colon \D\to\Delta$ and $\tau\colon \Omega_1\to\D$ 
be conformal mappings. These mappings have continuous extensions to the boundaries
so that we have a mapping $h:=\tau\circ f_1\circ\sigma\colon\partial \D\to\partial\D$.
We thus want to choose $H$ such that $h$ satisfies the hypotheses of Lemma~\ref{la10}.

Since $g_1$ has (simple) critical points at $0$ and $1+i\tau$ and
$f_1(z)=g_1(z)$ for $z\in\partial\Delta\cap S_1$,  
we find that this is the case 
if $H\in C^2[0,1]$ with $H'(0)=H'(1)=0$, $H''(0)\neq 0$,
$H''(1)\neq 0$, and $H'(x)>0$ for $0<x<1$.
So we fix any such mapping~$H$. We extend the mapping $h$ to a quasiconformal 
self-mapping of $\D$. The corresponding extension of the mapping
$f_1\colon\partial\Delta\to\partial\Omega_1$ is then given by
$f_1:=\tau^{-1}\circ h\circ\sigma^{-1}\colon\Delta\to\Omega_1$.

Next we define $f_1$ on $\Delta_{1,1}$. Again we define it first on $\partial \Delta_{1,1}$. 
As $f_1$ is defined on $\partial\Delta_{1,1}\cap S_1$ already,
we have to define it only on $\partial\Delta_{1,1}\cap \partial S_1 =\{t(1+i\tau)\colon 1\leq t\leq 2\}$.
We do so by putting $f_1(t(1+i\tau))=f_1((2-t)(1+i\tau))$ for $1\leq t\leq 2$.
As before
the mapping $f_1\colon\partial \Delta_{1,1}\to\bC$ 
can be extended quasiconformally to the interior of $\Delta_{1,1}$.

We have thus defined $f_1$ on $\Delta_{0,0}\cup\Delta_{1,1}$.
We extend the definition to $\Delta_{m,m}$ with $m\geq 2$ by periodicity; that 
is, we put $f_1(z)=f_1(z-2\lfloor m/2\rfloor(1+i\tau))$ for
$z\in\Delta_{m,m}$.

It remains to define $f_1$ in the still missing triangles in 
the lower half-plane, which are of the form $\Delta_{m,-m-1}$ with $m\in\N_0$.
This we do by reflection in the real axis; that is, we put 
$f_1(z)=\overline{f_1(\overline{z})}$ for $z$ in such a triangle.

We have thus defined the quasiregular mapping $f_1\colon S_1\to\bC$.
To define the quasiregular mapping $f_2\colon S_2\to\bC$, we put
\begin{equation} \label{k2}
g_2(z)
=-\overline{g_1(\overline{z}+1+i\tau)}.
\end{equation}
Note that $g_2$ maps the rectangle $R_{-1,0}$ onto
the upper half-plane and satisfies $g_2(0)=1$, $g_2(i\tau)=\infty$, $g_2(-1+i\tau)=-1$
and $g_2(-1)=-\overline{g_1(i\tau)}=-a$.

For $m,n\in\Z$  such that $R_{m,n}\subset S_2$ and $z\in\R_{m,n}$ we put $f_2(z)=g_2(z)$.
We also put $f_2(x)=g_2(x)$ for $-1< x\leq 0$
and define $f_2$ on $\partial S_2$ by~\eqref{k6}.
This defines $f_2$ on the sector $S_2$ except for the interior of the 
triangles $R_{-m,m-1}\cap S_2$ and $R_{-m,-m}\cap S_2$ with $m\in\N$.
As before we can extend $f_2$ quasiconformally to these triangles.
Here the triangle $R_{-m,m-1}$ is mapped onto the half-disk
$\Omega_2 :=\{z\in\C\colon \im z>0,\, |z|<1\}$ while 
$R_{-m,-m}$ is mapped onto $\overline{\Omega_2}$.

Thus for $j\in\{1,2\}$ we have defined a quasiregular mapping $f_j\colon S_j\to\bC$
such that~\eqref{k6} holds. This implies that the mapping $f_0$ defined by~\eqref{k7} is 
quasiregular.

By construction, all $(-1)$-points and all poles of $f_0$ are multiple, and 
except for the origin all $1$-points of $f_0$ are also multiple.
Moreover, all $a$-points in the right half-plane are multiple and all
$(-a)$-points in the left half-plane are multiple.
It follows from~\eqref{k6} and~\eqref{k1} that $f_0$ maps the imaginary axis to $\partial \D$.

Let $D$ be a Jordan domain which contains $a$ and $-a$ and 
whose closure is contained in~$\D$.
Then an island over $D$ cannot intersect the imaginary axis.
Thus all islands over $D$ are contained in the right or left half-plane.
Those contained in the right half-plane contain a multiple $a$-point
while those contained in the left half-plane contain a multiple $(-a)$-point.
We deduce that there are no simple islands over~$D$.

As mentioned above, Lemma~\ref{lemma-qr} yields that there exists a quasiconformal mapping
$\phi\colon\C\to\C$ such that $F:=f_0\circ \phi$ is meromorphic.
It is easy to see that the set $A$ of all $z$ satisfying $|z|>1$
where $f_1$ and $f_2$ are not meromorphic satisfies~\eqref{c4}.
Lemma~\ref{lemma-twb} now yields that $\phi$ satisfies~\eqref{c3}.
Since an elliptic function has order $2$ this implies that 
$\rho(F)=4\alpha/\pi$.

As it is the case for $f_0$, the function $F$ has no simple island over $D$ 
and all poles and all $(\pm 1)$-points of $F$ are multiple, except
for the simple $1$-point at the origin.
We finally put $f(z)=F(z^p)$ for some $p\in\N$ with $p\geq 2$.
Then the origin is a multiple $1$-point of $f$ and we conclude
that $f$ has no simple island over $D$ 
and that all poles and all $(\pm 1)$-points of $f$ are multiple.
Moreover, $\rho(f)=4p\alpha/\pi$.

Since $\alpha=\arctan\tau$ we can achieve $\rho(f)=\rho$ for any given
$\rho\in (0,\infty)$ by a suitable choice of $\tau$ and~$p$.
\end{proof}

\noindent
Mathematisches Seminar\\
Christian-Albrechts-Universit\"at zu Kiel\\
Ludewig-Meyn-Str.\ 4\\
24098 Kiel\\
Germany\\
{\tt Email: bergweiler@math.uni-kiel.de}

\medskip

\noindent
Department of Mathematics\\
Purdue University\\
West Lafayette, IN 47907\\
USA\\
{\tt Email: eremenko@math.purdue.edu}


\begin{thebibliography}{99}
\bibitem{Ahlfors1932a} 
L.\ V.\ Ahlfors,
Sur une g\'en\'eralisation du th\'eor\`eme de Picard.
{\rm C.~R.\ Acad.\ Sci.\ Paris} 194 (1932), 245--247;
{\rm Collected Papers}, Birkh\"auser, Boston, Basel, Stuttgart,
1982, Vol.~I, pp.\ 145--147.

\bibitem{Ahlfors1932b} 
L.\ V.\ Ahlfors,
Sur les fonctions inverses des fonctions m\'eromorphes.
{\rm C.~R.\ Acad.\ Sci.\ Paris} 194 (1932), 1145--1147;
{\rm Collected Papers}, Vol.~I, pp.\ 149--151.


\bibitem{Ahlfors1932c} 
L. V. Ahlfors,
Quelques propri\'et\'es des surfaces de Riemann correspondant aux fonctions
m\'e\-ro\-morphes.
Bull.\ Soc.\ Math.\ France 60 (1932), 197--207;
Collected Papers, Vol.~I, pp.\ 152--162.

\bibitem{Ahlfors1933} 
L.\ V.\ Ahlfors,
\"Uber die Kreise die von einer Riemannschen Fl\"ache schlicht
\"uberdeckt werden.
{\rm Comm.\ Math.\ Helv.} 5 (1933), 28--38;
{\rm Collected Papers}, Vol.~I, pp.\ 163--173.

\bibitem{Ahlfors1935} 
L.\ V.\ Ahlfors,
Zur Theorie der \"Uberlagerungsfl\"achen.
{\rm Acta Math.} 65 (1935), 157--194;
{\rm Collected Papers}, Vol.~I, pp.\ 214--251.

\bibitem{Axler2001} 
Sheldon Axler,  Paul Bourdon and Wade Ramey,
Harmonic function theory.
Second edition.
Graduate Texts in Mathematics, 137
Springer-Verlag, New York, 2001.

\bibitem{Bergweiler1990}
Walter Bergweiler,
On the growth rate of composite meromorphic functions.
Complex Variables Theory Appl. 14 (1990), no.~1--4, 187--196.

\bibitem{Bergweiler1998}
Walter Bergweiler,
A new proof of the Ahlfors five islands theorem.
J.\ Anal.\ Math.
76 (1998), 337--347.
\bibitem{Bergweiler2006}
Walter Bergweiler,
Bloch's principle.
Comput. Methods Funct. Theory
6 (2006), 77--108.

\bibitem{Bergweiler1984}
Walter Bergweiler, Gerhard Jank and Lutz Volkmann.
Wachstumsverhalten zusammengesetzter Funktionen.
Results Math. 7 (1984), no.~1, 35--53.

\bibitem{Ahlfors1956} 
A. Beurling and L. Ahlfors, 
The boundary correspondence under quasiconformal mappings.
Acta Math. 96 (1956), 125--142;
The collected works of Arne Beurling,
Birkh\"auser Boston, Inc., Boston, MA, 1989,
Vol.~I, pp.\ 231--248;
{\rm Collected Papers} (of Ahlfors), Vol.~II, pp.\ 104--121.

\bibitem{Bloch1925}
A.\ Bloch,
Quelques th\'eor\`emes sur les fonctions enti\`eres et m\'eromorphes d’une variable. 
{\rm C.\ R.\ Acad.\ Sci.\ Paris}
181 (1925), 1123--1125; 182 (1926), 367--369.

\bibitem{Bloch1926}
A.\ Bloch,
La conception actuelle de la th\'eorie des
fonctions enti\`eres et m\'eromorphes.
{\rm Enseign. Math.} 25 (1926), 83--103.

\bibitem{Boas1954}
Ralph Philip Boas, Jr.,
Entire functions. Academic Press, Inc., New York, 1954. 

\bibitem{Branner2014}
Bodil Branner and N\'uria Fagella, 
Quasiconformal surgery in holomorphic dynamics.
Cambridge Studies in Advanced Mathematics, 141.
Cambridge University Press, Cambridge, 2014.

\bibitem{deThelin2005}
Henry de Th\'elin,
Une d\'emonstration du th\'eor\`eme de recouvrement de surfaces d'Ahlfors. 
Enseign. Math. (2) 51 (2005), no. 3--4, 203--209. 

\bibitem{Duval2014}
Julien Duval,
Sur la th\'eorie d'Ahlfors des surfaces. 
Enseign. Math. (2) 60 (2014), no. 3--4, 417--420. 

\bibitem{Goldberg1963}
A. A. Goldberg and V. G. Tairova,
On entire functions with two finite completely ramified values (in Russian).
Zapiski Meh-mat. fakulteta Kharkovskogo gos. universiteta i Harkovskogo mat. obshchestva, 29, Ser. 4 (1963) 67--78.

\bibitem{Goldberg2008}
Anatoly A.\ Goldberg and Iossif V.\ Ostrovskii,
Value distribution of meromorphic functions.
Translations of Mathematical Monographs, 236. American Mathematical Society, Providence, RI, 2008.

\bibitem{Kelingos1966}
J. A. Kelingos,
Boundary correspondence under quasiconformal mappings.
Michigan Math. J. 13 (1966), 235--249.

\bibitem{Kuenzi1956}
Hans K\"unzi,
Zur Theorie der Viertelsenden Riemannscher Fl\"achen.
Comment. Math. Helv. 30 (1956), 107--115.

\bibitem{Lehto1973}
O. Lehto and K. I. Virtanen,
Quasiconformal mappings in the plane.
Second edition.
Die Grund\-lehren der mathematischen Wissenschaften, 126.
Springer-Verlag, New York, Heidelberg, 1973.

\bibitem{Nevanlinna1929}
Rolf Nevanlinna,
Le th\'eor\`eme de Picard-Borel et la th\'eorie des fonctions m\'eromorphes.
Gauthiers-Villars, Paris, 1929.

\bibitem{Pommerenke1992}
Ch. Pommerenke,
Boundary behaviour of conformal maps.
Die Grund\-lehren der Mathematischen Wissenschaften, 299.
Springer-Verlag, Berlin, 1992.

\bibitem{Selberg1928} Henrik L. Selberg, 
\"Uber einige Eigenschaften bei der Werteverteilung der meromorphen
Funktionen endlicher Ordnung.
Avh. Det Norske Videnskaps-Akademi i Oslo I.  Matem.-Naturvid. Kl., 7 (1928) 17pp.

\bibitem{Teichmueller1944} 
Oswald Teichm\"uller,
Einfache Beispiele zur Wertverteilungslehre.
Deut\-sche Math. 7 (1944), 360--368;
Gesammelte Abhandlungen -- Collected Papers.
Springer-Verlag, Berlin, Heidelberg, New York, 1982, pp.~728--736. 


\bibitem{Toki1957}
Yukinari T\^{o}ki,
Proof of Ahlfors principal covering theorem.
Rev. Math. Pures Appl. 2 (1957), 277--280.

\bibitem{Valiron1923} Georges Valiron,
Lectures on the general theory of integral functions.
\'Edouard Privat, Toulouse 1923; reprint: Chelsea, New York, 1949.

\bibitem{Walsh1939}
J. L. Walsh,
Note on the location of zeros of the derivative of a rational function
whose zeros and poles are symmetric in a circle.
 Bull. Amer. Math. Soc. 45 (1939), 462--470.

\bibitem{Zalcman1968}
Lawrence Zalcman,
Analytic capacity and rational approximation.
Lecture Notes in Mathematics, 50.
Springer-Verlag, Berlin, New York, 1968. 

\bibitem{Zalcman1975}
Lawrence Zalcman,
A heuristic principle in complex function theory.
{\rm Amer.\ Math.\ Monthly}
82 (1975), 813--817.

\bibitem{Zalcman1998}
Lawrence Zalcman,
Normal families: new perspectives.
{\rm Bull.\ Amer.\ Math.\ Soc.\ (N.~S.)}
35 (1998), 215--230.
\end{thebibliography}
\end{document}